\documentclass[12pt, reqno]{amsart}
\usepackage{biblatex}
\usepackage{hyperref}
\usepackage{aliascnt}
\usepackage{amssymb}
\usepackage{mathtools}
\usepackage{xifthen}
\usepackage{array}
\usepackage{float}
\usepackage{multirow}
\usepackage{tikz,tikz-cd}
\usetikzlibrary{arrows.meta,bending}
\hypersetup{
    colorlinks,
    linkcolor={red},
    citecolor={magenta},
    urlcolor={blue}
}

\makeatletter
\@namedef{subjclassname@2020}{\textup{2020} Mathematics Subject Classification}
\makeatother

\newtheorem{thm}{Theorem}[section]

\newaliascnt{prop}{thm}
\newtheorem{prop}[prop]{Proposition}
\aliascntresetthe{prop}

\newaliascnt{cor}{thm}
\newtheorem{cor}[cor]{Corollary}
\aliascntresetthe{cor}

\theoremstyle{definition}
\newaliascnt{defn}{thm}
\newtheorem{defn}[defn]{Definition}
\aliascntresetthe{defn}
\newaliascnt{example}{thm}
\newtheorem{example}[example]{Example}
\aliascntresetthe{example}

\newaliascnt{prob}{thm}
\newtheorem{prob}[prob]{Problem}
\aliascntresetthe{prob}

\providecommand{\PP}{\mathbb{P}} 
\providecommand{\N}{\mathbb{N}} 
\providecommand{\Z}{\mathbb{Z}} 
\providecommand{\R}{\mathbb{R}} 
\providecommand{\F}{\mathbb{F}} 
\DeclareMathOperator{\dom}{Dom} 
\providecommand{\set}[2][]{
	\ifthenelse{\isempty{#1}}{
		\left\{#2\right\}
	}{
		\left\{\,#1\;\middle|\;#2\,\right\}}
	}
\DeclareMathOperator{\pow}{Sb} 
\newcommand{\abs}[1]{\left\lvert#1\right\rvert} 
\DeclareMathOperator{\orb}{Orb} 
\DeclareMathOperator{\alt}{Alt} 
\DeclareMathOperator{\perm}{Perm} 
\DeclareMathOperator{\im}{Im} 
\DeclareMathOperator{\aut}{Aut} 
\DeclareMathOperator{\spn}{Span} 
\providecommand{\dd}{\mathrm{d}} 
\providecommand{\diff}[2]{\frac{\dd#1}{\dd#2}} 
\providecommand{\pdif}[2]{\frac{\partial#1}{\partial#2}} 
\providecommand{\integ}[4]{\int_{#1}^{#2}#3\,\dd#4} 
\DeclareMathOperator{\setcat}{Set} 

\DeclareMathOperator{\vertices}{Vert} 
\DeclareMathOperator{\scmplx}{SCmplx} 
\DeclareMathOperator{\fct}{Fct} 

\DeclareMathOperator{\pmfld}{PMfld} 
\DeclareMathOperator{\simpcmplx}{Sim} 

\DeclareMathOperator{\magm}{Mag} 
\DeclareMathOperator{\quasi}{Quas} 
\DeclareMathOperator{\aq}{AQ} 
\DeclareMathOperator{\ncaq}{NCAQ} 
\DeclareMathOperator{\nct}{NCT} 
\DeclareMathOperator{\inp}{In} 
\DeclareMathOperator{\out}{Out} 
\DeclareMathOperator{\free}{Fr} 
\DeclareMathOperator{\cg}{Cg} 

\DeclareMathOperator{\topsp}{Top} 
\DeclareMathOperator{\mfld}{Mfld} 
\DeclareMathOperator{\smfld}{SMfld} 
\DeclareMathOperator{\geo}{Geo} 
\DeclareMathOperator{\ogeo}{OGeo} 
\DeclareMathOperator{\cvx}{Cvx} 
\DeclareMathOperator{\ocvx}{OCvx} 
\DeclareMathOperator{\bipyr}{Bipyr} 
\DeclareMathOperator{\oser}{OSer} 
\DeclareMathOperator{\ncvert}{NCVert} 
\DeclareMathOperator{\ncedge}{NCEdge} 
\DeclareMathOperator{\ncgr}{NCGrph} 
\DeclareMathOperator{\ser}{Ser} 
\DeclareMathOperator{\cmplt}{Cmplt} 
\DeclareMathOperator{\eucmplt}{EuCmplt} 
\DeclareMathOperator{\riem}{Riem} 

\begin{document}
\title{Orientable triangulable manifolds are essentially quasigroups}
\author[C. Aten]{Charlotte Aten}
\address{Department of Mathematics\\
University of Denver\\Denver 80208\\USA}
\urladdr{\href{https://aten.cool}{https://aten.cool}}
\email{\href{mailto:charlotte.aten@du.edu}{charlotte.aten@du.edu}}
\author[S. Yoo]{Semin Yoo}
\address{School of Computational Sciences\\
Korea Institute for Advanced Study\\Seoul 02455\\South Korea}
\urladdr{\href{http://yooso.cool}{http://yooso.cool}}
\email{\href{mailto:syoo19@kias.re.kr}{syoo19@kias.re.kr}}
\thanks{The second author is supported by the KIAS Individual Grant (CG082701) at
Korea Institute for Advanced Study. Thanks to Jonathan Smith for providing the quasigroup in \autoref{ex:field} and thanks to Jonathan Pakianathan for suggesting this topic to the second author as a possible subject for her PhD thesis.}

\begin{abstract}
We introduce an \(n\)-dimensional analogue of the construction of tessellated surfaces from finite groups first described by Herman and Pakianathan. Our construction is functorial and associates to each \(n\)-ary alternating quasigroup both a smooth, flat Riemannian \(n\)-manifold which we dub the open serenation of the quasigroup in question, as well as a topological \(n\)-manifold (the serenation of the quasigroup) which is a subspace of the metric completion of the open serenation. We prove that every connected orientable smooth manifold is serene, in the sense that each such manifold is a component of the serenation of some quasigroup. We prove some basic results about the variety of alternating \(n\)-quasigroups and note connections between our construction, Latin hypercubes, and Johnson graphs.
\end{abstract}

\maketitle

\tableofcontents

\section{Introduction}
This work builds on the ideas behind the construction of tessellated surfaces from finite groups given by Herman and Pakianathan \cite{herman}. They produced a functor from a category of nonabelian finite groups equipped with certain homomorphisms preserving noncommutativity to a category of singular, oriented \(2\)-manifolds. They then desingularized these manifolds into compact orientable surfaces. The group used in this construction (or a subquotient, or its automorphism group) would then act on the cells of a functorially-occurring polyhedral tessellation of the resulting manifold, yielding a faithful, orientation-preserving group action. Herman and Pakianathan's motivation was the manufacture of these actions and the study of finite groups through these associated surfaces.

Our motivation is the exact opposite. In higher dimensions the orientable manifolds are much less well understood so we reverse course by showing that the same method allows us to understand higher-dimensional manifolds by working with certain discrete algebraic structures. In order to extend this technique to dimensions \(3\) and above a radically different formalism is needed, although we will see in \autoref{sec:serenation} that when applied to a finite group it produces the same results Herman and Pakianathan saw. We hope that the reader will agree that it is morally the correct generalization. The appropriate analogue of a group must be found (which is not an \(n\)-ary group in the sense Post introduced in \cite{post}, as one might have guessed) and one must find a new technique for desingularization of singular \(n\)-manifolds. We supply these concepts.

We construct smooth \(n\)-dimensional manifolds from \(n\)-ary analogues of quasigroups. An \(n\)-quasigroup consists of a set \(A\) and a multiplication operation \(f\colon A^n\to A\) such that the equation
    \[
        f(x_1,\dots,x_n)=x_{n+1}
    \]
has a unique solution whenever all but one of the \(x_i\) are fixed. This is equivalent to the graph
    \[
        f^\square=\set[(x_1,\dots,x_n,f(x_1,\dots,x_n))\in A^{n+1}]{x_1,\dots,x_n\in A}
    \]
being a Latin \(n\)-cube in the sense that the first \(n\)-coordinates give the location of a cell and the final coordinate is its label. A Latin \(2\)-cube is the familiar Latin square from combinatorics. For each \(n\) we produce an \emph{open serenation functor} \(\oser_n\) from a wide subcategory \(\ncaq_n\) of a variety of \(n\)-quasigroups to a category of smooth \(n\)-manifolds. For us, a class of \(n\)-quasigroups we call \emph{alternating} will play the role of nonabelian groups in Herman and Pakianathan's work. Given any alternating \(n\)-quasigroup \(A\) we will find that \(\oser_n(A)\) is orientable and consists of a (perhaps uncountable) collection of second countable connected components. Moreover, \(\oser_n(A)\) inherits a metric from the \(n\)-quasigroup structure on \(A\) which makes \(\oser_n(A)\) into a Riemannian manifold. Certain homomorphisms \(A\to B\) of \(n\)-quasigroups yield smooth maps \(\oser_n(A)\to\oser_n(B)\) which are local isometries everywhere with respect to these metrics.

We can use this metric structure to manufacture another functor, the \emph{serenation functor} denoted by \(\ser_n\), from the same category of \(n\)-quasigroups to a category of topological manifolds. This functor is in a sense the metric completion of \(\oser_n\). Although in dimensions \(4\) and above we can't upgrade the codomain of this functor to again be a category of smooth manifolds, it does have a redeeming property, which is that every connected orientable triangulable topological \(n\)-manifold is a component of the serenation of a quasigroup. This is \autoref{thm:orientable_triangulable}, which has as an immediate corollary \autoref{cor:smooth_manifold}, which says that every connected orientable smooth \(n\)-manifold is a component of the serenation of an alternating \(n\)-quasigroup.

As indicated by the mention of a subcategory of a variety of \(n\)-quasigroups above, the term \emph{variety} here (and throughout this paper) is being used in the sense of universal algebra rather than algebraic geometry. A \emph{variety} (or \emph{variety of (universal) algebras}) is a class of algebras closed under taking homomorphic images, products, and subalgebras, which is equivalent by Birkhoff's Theorem to saying that a variety is a class of algebras defined by universally-quantified equations. For example, abelian groups form a variety, but the class of nonabelian groups is not a variety. Every variety of algebras is a category whose morphisms are homomorphisms of algebras. The varieties of \(n\)-quasigroups which we utilize have not to our knowledge been studied before. However, there is a long history of interplay between certain varieties of \(n\)-quasigroups and both combinatorics and topology. On the combinatorial side, varieties of \(n\)-quasigroups often correspond to certain nice classes of Latin squares or Latin (hyper)cubes \cite{chaffer,csima}. Even when the algebraic viewpoint is not present, as in \cite{kuhl}, we will see in \autoref{sec:latin_cubes} that we encounter very similar problems in our setting. Historically the study of loops (quasigroups with identity) and finite projective planes have been closely intertwined \cite{albert}. Recently there has been a great deal of interplay between ternary quasigroups and knot theory \cite{niebrzydowski}. The present work may be taken as evidence that the aforementioned connection between quasigroups and knot theory is an example of a general pattern in which topology and quasigroup theory inform each other, rather than a special quirk of the theory of knots.

Our \autoref{prob:finiteness_compactness} asks whether it is possible to realize every connected compact orientable triangulable manifold as a homeomorphic copy of a component of the serenation of some finite alternating quasigroup. An affirmative answer to this question would be implied by that to a combinatorial question, which is our \autoref{prob:evans} and which is a generalized version of the Evans conjecture about the completion of partial Latin squares.

We are concerned with the actions of the symmetric and alternating groups on \(n\)-tuples from a set \(A\). Here we give our conventions in this direction. We take \(\PP=\set{1,2,3,\dots}\) and \(\N=\set{0,1,2,\dots}\). Given \(n\in\PP\) we define \([n]=\set{1,2,3,\dots,n}\). Given a set \(S\) we denote by \(\perm_S\) (or \(\perm(S)\)) the group of permutations of \(S\) and we denote by \(\alt_S\) the group of even permutations of \(S\). When \(S=[n]\) for some \(n\in\PP\) we write \(\perm_n\) and \(\alt_n\) rather than \(\perm_{[n]}\) and \(\alt_{[n]}\), respectively. Given a group \(G\) and an action \(h\colon G\to\perm_S\) of \(G\) on a set \(S\) we denote by \(\orb(h)\) the collection of orbits of \(S\) under the action \(h\) and we denote by \(\orb_h(s)\) the orbit of \(s\in S\) under the action \(h\). When \(h\) is understood by context we write \(\orb_{G}(S)\) (or even just \(\orb(S)\)) instead of \(\orb(h)\). Similarly we write \(\orb_{G}(s)\) (or even just \(\orb(s)\)) instead of \(\orb_h(s)\) for \(s\in S\) when context makes the action clear. We denote by \(\pow(S)\) the collection of all finite subsets of the set \(S\) and we denote by \(\abs{S}\) the cardinality of the set \(S\). We write \(A\le B\) to indicate that an algebra \(A\) (such as a group) is a subalgebra of the algebra \(B\) (which is necessarily of the same signature).

Here a \emph{manifold} is a finite-dimensional real manifold without boundary which, while not necessarily second countable, is taken to consist of a collection (whose cardinality is unconstrained) of second countable connected components. A \emph{smooth manifold} is a manifold in the previous sense equipped with a smooth atlas. We denote by \(\mfld_n\) and \(\smfld_n\) the categories of \(n\)-dimensional manifolds and \(n\)-dimensional smooth manifolds, respectively. A \emph{Riemannian manifold} is a pair \((M,g)\) where \(M\) is a smooth manifold in the previous sense and \(g\) is a smooth Riemannian metric on \(M\). We write \(\riem_n\) to indicate the category whose objects are \(n\)-dimensional Riemannian manifolds and whose morphisms are smooth maps which are local isometries everywhere.

We outline here the structure of the paper. In \autoref{sec:pseudomanifolds} we give some preliminaries on simplicial complexes and pseudomanifolds. We do the same for algebraic structures in \autoref{sec:algebraic_prelim}. We introduce the algebras of primary concern for us, which are alternating quasigroups, in \autoref{sec:alternating_quasigroups}. With the basic algebra and topology in place, we define and see examples of open serenation, our Riemannian manifold construction, in \autoref{sec:open_serenation}. The following section, \autoref{sec:serenation}, details the serenation functor, which is a sort of metric completion of the open serenation functor. It is in this section that we prove our \autoref{thm:orientable_triangulable}, which says that all connected orientable triangulable topological manifolds may be produced via this construction. In \autoref{sec:latin_cubes} we examine similar questions to those answered in the preceding section, but with finiteness and compactness constraints added.

\section{Pseudomanifolds}
\label{sec:pseudomanifolds}
In this section we detail the basic tools pertaining to pseudomanifolds which we will use in our construction. We refer to Munkres \cite{munkres} for basic topological definitions but we deviate from his terminology in several places.

\subsection{Simplicial complexes}
Pseudomanifolds are a family of particularly well-behaved simplicial complexes, so we begin by discussing this latter class of combinatorial objects.

\begin{defn}[Simplicial complex]
A \emph{simplicial complex} is a set of sets \(\Gamma\) such that for each \(\gamma\in\Gamma\) and each \(\gamma'\subseteq\gamma\) we have that \(\gamma'\in\Gamma\).
\end{defn}

Our simplicial complexes are elsewhere known as \emph{abstract simplicial complexes}.

\begin{defn}[Vertices of a simplicial complex]
Given a simplicial complex \(\Gamma\) we say that a member of a singleton set \(\gamma\subset\Gamma\) is a \emph{vertex} of \(\Gamma\). We write \(\vertices(\Gamma)\) to denote the set of vertices of the simplicial complex \(\Gamma\).
\end{defn}

We have terminology for general members \(\gamma\) of a simplicial complex \(\Gamma\), as well.

\begin{defn}[Face]
Given a simplicial complex \(\Gamma\) we refer to an element \(\gamma\in\Gamma\) as a \emph{face} of \(\Gamma\). Given two faces \(\gamma,\gamma'\in\Gamma\) we say that \(\gamma'\) is a \emph{face} of \(\gamma\) when \(\gamma'\subseteq\gamma\).
\end{defn}

Our simplicial complexes always have the empty set as a face.

\begin{defn}[Simplicial map]
Given simplicial complexes \(\Gamma_1\) and \(\Gamma_2\) we say that a function \(h\colon\vertices(\Gamma_1)\to\vertices(\Gamma_2)\) is a \emph{simplicial map} from \(\Gamma_1\) to \(\Gamma_2\) and write \(h\colon \Gamma_1\to\Gamma_2\) when for each \(\gamma\in\Gamma_1\) we have that \(\set[h(s)\in\vertices(\Gamma_2)]{s\in\gamma}\) is a face of \(\Gamma_2\).
\end{defn}

Simplicial maps are the morphisms in the category associated to simplicial complexes.

\begin{defn}[Category of simplicial complexes]
The \emph{category of simplicial complexes} \(\scmplx\) is the category whose objects are simplicial complexes, whose morphisms are simplicial maps, and whose identities and composition are those induced by the underlying functions of those simplicial maps.
\end{defn}

We will fix a positive integer \(n\) and focus our attention on complexes whose facets are in the following sense \(n\)-dimensional.

\begin{defn}[Dimension of a face]
Given a simplicial complex \(\Gamma\) and a face \(\gamma\in\Gamma\) with \(\abs{\gamma}\in\N\) we say that the \emph{dimension} of \(\gamma\) is \(\abs{\gamma}-1\), that \(\gamma\) is an \emph{(\(\abs{\gamma}-1)\)-face}, or that \(\gamma\) is \emph{(\(\abs{\gamma}-1\))-dimensional}.
\end{defn}

In particular this means that every simplicial complex has exactly one \((-1)\)-dimensional face, the empty face. On the opposite end of the size spectrum we have facets.

\begin{defn}[Facet]
Given a simplicial complex \(\Gamma\) we say that a face \(\gamma\in\Gamma\) is a \emph{facet} of \(\Gamma\) when \(\gamma\) is maximal with respect to subset inclusion among faces of \(\Gamma\).
\end{defn}

\begin{defn}[Facets of a simplicial complex]
We denote the set of all facets of a simplicial complex \(\Gamma\) by \(\fct(\Gamma)\).
\end{defn}

It is possible that a simplicial complex has no facets or that a simplicial complex has some facets but there are faces which are not contained in any facet. These pathologies won't arise in the cases we actually consider.

\begin{defn}[Pure simplicial complex]
We say that a simplicial complex \(\Gamma\) is \emph{pure} when
    \begin{enumerate}
        \item given any \(\gamma,\gamma'\in\fct(\Gamma)\) we have that \(\abs{\gamma}=\abs{\gamma'}\) and
        \item given any \(\gamma\in\Gamma\) there exists some \(\gamma'\in\fct(\Gamma)\) such that \(\gamma\subseteq\gamma'\).
    \end{enumerate}
\end{defn}

\begin{defn}[Dimension of a pure simplicial complex]
When \(\Gamma\) is a pure simplicial complex such that each \(\gamma\in\fct(\Gamma)\) is \(n\)-dimensional we say that \(\Gamma\) is an \emph{\(n\)-dimensional simplicial complex}.
\end{defn}

\subsection{Pseudomanifolds}
Pseudomanifolds are an even more special class of simplicial complexes contained within the class of pure simplicial complexes. Intuitively, \(n\)-pseudomanifolds look like \(n\)-dimensional manifolds, except that their \((n-2)\)-skeleta may be highly singular.

\begin{defn}[Pseudomanifold]
We say that an \(n\)-dimensional simplicial complex \(\Gamma\) is an \emph{\(n\)-pseudomanifold} when given an \((n-1)\)-face \(\gamma_1\in\Gamma\) there exist exactly two facets \(\gamma_2,\gamma_3\in\Gamma\) such that \(\gamma_2\cap\gamma_3=\gamma_1\).
\end{defn}

We want to consider morphisms of pseudomanifolds which respect the pseudomanifold structure, so general simplicial maps will not suffice.

\begin{defn}[Pseudomanifold morphism]
Given pseudomanifolds \(\Gamma_1\) and \(\Gamma_2\) we say that a simplicial map \(h\colon\Gamma_1\to\Gamma_2\) is a \emph{morphism of pseudomanifolds} when for each facet \(\gamma\) of \(\Gamma_1\) we have that \(h(\gamma)\) is a facet of \(\Gamma_2\).
\end{defn}

Pseudomanifolds and their morphisms form a category.

\begin{defn}[Category of \(n\)-pseudomanifolds]
We refer to the subcategory of \(\scmplx\) whose objects are \(n\)-pseudoomanifolds and whose morphisms are all pseudomanifold morphisms as the \emph{category of \(n\)-pseudomanifolds}, which we denote by \(\pmfld_n\).
\end{defn}

\subsection{Geometric realization}
\label{subsec:geometric_realization}
We collect here the terminology we need about geometric realization of simplicial complexes. We denote by \(\topsp\) the category whose objects are spaces, whose morphisms are continuous maps, whose identities are the usual identity functions, and whose composition is the usual composition of functions.

Given a set \(S\) we denote by \(\R^S\) the \(S\)-fold Cartesian power of the real line equipped with the Euclidean topology. Given \(X\subseteq\R^S\) we denote by \(\cvx(X)\) the (closed) convex hull of \(X\) and by \(\ocvx(X)\) the interior of \(\cvx(X)\) in the affine hull of \(X\). Given any \(\gamma\subseteq S\) we identify each member of \(\gamma\) with the corresponding basis element of \(\R^S\) so that \(\ocvx(\gamma)\) is the interior (in the appropriate affine subspace) of the simplex in \(\R^S\) whose vertices form the set \(\gamma\).

We make use of a nonstandard geometric realization of pseudomanifolds, in the sense that this geometric realization is not the restriction of the usual geometric realization of simplicial complexes to the category of \(n\)-pseudomanifolds. Our geometric realization of pseudomanifolds drops the \((n-2)\)-skeleton of a pseudomanifold in order to dispose of the singularities present there.

\begin{defn}[Open geometric realization functor]
Fix \(n\in\PP\). The \emph{open geometric realization functor}
    \[
        \ogeo_n\colon\pmfld_n\to\topsp
    \]
is defined as follows. Given an \(n\)-pseudomanifold \(\Gamma\) we define \(\ogeo_n(\Gamma)\subseteq\R^{\vertices(\Gamma)}\) to be the set
    \[
        \ogeo_n(\Gamma)=\bigcup_{\mathclap{\substack{\gamma\in\Gamma\\\abs{\gamma}\in\set{n,n+1}}}}\ocvx(\gamma)
    \]
equipped with the subspace topology which \(\ogeo_n(\Gamma)\) inherits from \(\R^{\vertices(\Gamma)}\).

Given an \(n\)-pseudomanifold morphism \(h\colon\Gamma_1\to\Gamma_2\) we define
    \[
        \ogeo_n(h)\colon\ogeo(\Gamma_1)\to\ogeo_n(\Gamma_2)
    \]
by the rule
    \[
        \ogeo_n(h)\left(\sum_{s\in\gamma}\alpha_s s\right)=\sum_{s\in\gamma}\alpha_sh(s)
    \]
when \(\gamma\in\Gamma_1\) and \(\sum_{s\in\gamma}\alpha_s s\in\ocvx(\gamma)\).
\end{defn}

For a general simplicial complex we have the usual (closed) geometric realization functor.

\begin{defn}[Geometric realization functor]
The \emph{geometric realization functor}
    \[
        \geo\colon\simpcmplx\to\topsp
    \]
is defined as follows. Given a simplicial complex \(\Gamma\) we define \(\geo(\Gamma)\subseteq\R^{\vertices(\Gamma)}\) to be the set
    \[
        \geo(\Gamma)=\bigcup_{\gamma\in\Gamma}\cvx(\gamma)
    \]
equipped with the subspace topology which \(\geo(\Gamma)\) inherits from \(\R^{\vertices(\Gamma)}\).

Given a simplicial map \(h\colon\Gamma_1\to\Gamma_2\) we define
    \[
        \geo(h)\colon\geo(\Gamma_1)\to\geo(\Gamma_2)
    \]
by the rule
    \[
        \geo_n(h)\left(\sum_{s\in\gamma}\alpha_s s\right)=\sum_{s\in\gamma}\alpha_sh(s)
    \]
when \(\gamma\in\Gamma_1\) and \(\sum_{s\in\gamma}\alpha_s s\in\cvx(\gamma)\).
\end{defn}

\section{Algebraic preliminaries}
\label{sec:algebraic_prelim}
The inputs in our construction are a family of algebras equipped with a particular class of homomorphisms between them. In this section we set up all the pertinent algebraic machinery.

\subsection{Magmas}
We introduce the basic notation and terminology we will need with respect to magmas. We use Bergman \cite{bergman}, Smith and Romanowska \cite{smith}, and Burris and Sankpanavar \cite{burris} as general references for universal algebra.

\begin{defn}[Magma]
Given some \(n\in\PP\) we refer to a set \(A\) equipped with an \(n\)-ary operation \(f\colon A^n\to A\) as a \emph{magma} (or as an \emph{\(n\)-magma} when we want to emphasize the arity of \(f\)). We sometimes write \((A,f)\) rather than \(A\) when we want to give an explicit name to the \(n\)-ary operation when introducing a magma. Otherwise, we assume the \(n\)-ary operation of \(A\) is denoted by \(f\).
\end{defn}

Note that this definition is broader than what is traditionally meant by ``binar'', ``groupoid'', or ``magma'' (in the sense of Bourbaki), which all refer to only the case where \(f\colon A^2\to A\) is a binary operation.

Magma homomorphisms are defined analogously to those for more familiar algebras.

\begin{defn}[Magma homomorphism]
Given \(n\)-magmas \((A,f)\) and \((B,g)\) we say that a function \(h\colon A\to B\) is a \emph{homomorphism} from \(A\) to \(B\) when for all \(a_1,\dots,a_n\in A\) we have that
    \[
        h(f(a_1,\dots,a_n))=g(h(a_1),\dots,h(a_n)).
    \]
We write \(h\colon A\to B\) to indicate that \(h\) is a homomorphism from \(A\) to \(B\).
\end{defn}

Magmas are the objects of a category whose morphisms are magma homomorphisms.

\begin{defn}[Category of magmas]
The \emph{category of \(n\)-magmas} \(\magm_n\) is the category whose objects are \(n\)-magmas, whose morphisms are \(n\)-magma homomorphisms, and whose identities and composition are those induced by the underlying functions of those homomorphisms.
\end{defn}

As is standard in universal algebra, we fix a signature \(\rho\colon I\to\N\) and write \(t_1(x_1,\dots,x_n)\approx t_2(x_1,\dots,x_n)\) (or \(t_1(x)\approx t_2(x)\), or even \(t_1\approx t_2\)) to indicate the formal expression (or \emph{identity})
    \[
        (\forall x_1,\dots,x_n)(t_1(x_1,\dots,x_n)=t_2(x_1,\dots,x_n))
    \]
where \(t_1\) and \(t_2\) are terms obtained by formally composing basic operation symbols \(\{f_i\}_{i\in I}\) where the arity of \(f_i\) is \(\rho(i)\).

Given an algebra \(A\) of signature \(\rho\) we write \(f_i^A\) to indicate the \(i^{\text{th}}\) basic operation of \(A\). We say that \(A\) \emph{models} \(t_1(x)\approx t_2(x)\) when
    \[
        (\forall x_1,\dots,x_n\in A)(t_1^A(x_1,\dots,x_n)=t_2^A(x_1,\dots,x_n))
    \]
where \(t_j^A\) is the term \(t_j\) viewed as an actual operation on \(A\) obtained by replacing each \(f_i\) appearing in \(t_j\) by the basic operation \(f_i^A\) of \(A\). In the case that \(A\) models \(t_1\approx t_2\) we write \(A\models t_1\approx t_2\).

We associate to each \(n\)-magma a particular group of permutations of \([n]\).

\begin{defn}[Permutomorphism]
Given a magma \(A\) we say that a permutation \(\alpha\in\perm_n\) is a \emph{permutomorphism} of \(A\) when
    \[
        A\models f(x_1,\dots,x_n)\approx f(x_{\alpha(1)},\dots,x_{\alpha(n)}).
    \]
\end{defn}

\begin{defn}[Permutomorphism group]
We denote by \(\perm(f)\) the group of all permutomorphisms of the magma \((A,f)\).
\end{defn}

We write \(\aut(A)\) to indicate the automorphism group of \((A,f)\), which is in general a different group from the permutomorphism group \(\perm(f)\) of \((A,f)\). Elements of \(\aut(A)\) are permutations of \(A\) while elements of \(\perm(A)\) are permutations of \([n]\).

We are particularly interested in the cases where \(\perm(f)\) is either \(\perm_n\) or \(\alt_n\).

\begin{defn}[Commutative magma]
We say that an \(n\)-magma \((A,f)\) is \emph{commutative} when \(\perm(f)=\perm_n\).
\end{defn}

\begin{defn}[Alternating magma]
We say that an \(n\)-magma \((A,f)\) is \emph{alternating} when \(\perm(f)\ge\alt_n\).
\end{defn}

It will turn out that most of the examples of alternating magmas we care to look at do have \(\perm(f)=\alt_n\), but we only require that \(\alt_n\) is a subgroup of \(\perm(f)\) in order to say that \((A,f)\) is alternating. In particular, every commutative magma is alternating.

\subsection{Quasigroups}
Intuitively, quasigroups are magmas in which division is always possible. We recall the definition of an \(n\)-quasigroup. Our definition is in analogy with Birkhoff's equational axioms for binary quasigroups \cite[p.5]{bergman}.

\begin{defn}[Quasigroup]
Given \(n\in\PP\) we say that an algebra
    \[
        (A,f,g_1,\dots,g_n)
    \]
of signature \((n,\dots,n)\) is a \emph{quasigroup} (or an \emph{\(n\)-quasigroup} when we want to emphasize the arity of \(f\)) when for each \(i\in[n]\) we have that \[A\models f(x_1,\dots,x_{i-1},g_i(x_1,\dots,x_{i-1},x_{i+1},\dots,x_n,y),x_{i+1},\dots,x_n)\approx y\] and \[A\models g_i(x_1,\dots,x_{i-1},x_{i+1},\dots,x_n,f(x_1,\dots,x_n))\approx x_i.\]
\end{defn}

It follows immediately from the definition of an \(n\)-quasigroup that for each \(n\in\PP\) we have that the class of \(n\)-quasigroups is an equational class and hence a variety.

We think of \(f\) as the multiplication of the \(n\)-quasigroup \((A,f,g_1,\dots,g_n)\) and we think of \(g_i\) as the \(i^{\text{th}}\)-coordinate division operation. That is, \[g_i(x_1,\dots,x_{i-1},x_{i+1},\dots,x_n,y)\] is taken to indicate the division of \(y\) simultaneously by \(x_j\) in the \(j^{\text{th}}\) coordinate for each \(j\neq i\). The impetus for this will be made clear by the following characterization of \(n\)-quasigroups.

\begin{defn}[\(n\)-quasigroup magma]
Given \(n\in\PP\) we say that an \(n\)-magma \(A\) is an \emph{\(n\)-quasigroup magma} when given
    \[
        x_1,\dots,x_{i-1},x_{i+1},\dots,x_n,y\in A
    \]
there exists a unique \(x_i\in A\) such that \(f(x_1,\dots,x_n)=y\).
\end{defn}

The \(n\)-quasigroups are in bijective correspondence with the \(n\)-quasigroup magmas, allowing us to freely switch between these two definitions when it suits us.

\begin{prop}
Given an \(n\)-quasigroup \((A,f,g_1,\dots,g_n)\) the \(n\)-magma reduct \((A,f)\) is an \(n\)-quasigroup magma. Conversely, given an \(n\)-quasigroup magma \((A,f)\) there exists a unique \(n\)-quasigroup expansion
    \[
        (A,f,g_1,\dots,g_n)
    \]
of \((A,f)\). In this expansion \(g_i\colon A^n\to A\) is defined so that
    \[
        g_i(x_1,\dots,x_{i-1},x_{i+1},\dots,x_n,y)
    \]
is the unique \(x_i\in A\) such that \(f(x_1,\dots,x_n)=y\).
\end{prop}

\begin{proof}
Suppose that \((A,f,g_1,\dots,g_n)\) is an \(n\)-quasigroup. We show that \((A,f)\) is an \(n\)-quasigroup magma. Suppose that \(f(x_1,\dots,x_n)=y\) for some \(x_1,\dots,x_n,y\in A\). It suffices to show that \(x_i=g_i(x_1,\dots,x_{i-1},x_{i+1},\dots,x_n,y)\) as this implies that \(x_i\) is uniquely determined by the \(x_j\) for \(j\neq i\) and \(y\). Observe that
    \begin{align*}
        x_i &= g_i(x_1,\dots,x_{i-1},x_{i+1},\dots,x_n,f(x_1,\dots,x_n)) \\
        &= g_i(x_1,\dots,x_{i-1},x_{i+1},\dots,x_n,y),
    \end{align*}
as claimed.

On the other hand, consider an \(n\)-quasigroup magma \((A,f)\). Suppose that
    \[
        (A,f,g_1,\dots,g_n)
    \]
is an \(n\)-quasigroup expansion of \(B\). If \(f(x_1,\dots,x_n)=y\) for some \(x_1,\dots,x_n,y\in A\) we must have that
    \[
        g_i(x_1,\dots,x_{i-1},x_{i+1},\dots,x_n,y)=x_i
    \]
by the \(n\)-quasigroup axioms. As there is a unique \(x_i\) for which \(f(x_1,\dots,x_n)=y\) the operation \(g_i\) is well-defined.

It remains to show that \(A\) with the \(g_i\) so defined does in fact satisfy all the \(n\)-quasigroup axioms. Given \(x_1,\dots,x_{i-1},x_{i+1},\dots,x_n,y\in A\) let \[x_i= g_i(x_1,\dots,x_{i-1},x_{i+1},\dots,x_n,y).\] By definition \(x_i\) is the unique element of \(A\) such that \(f(x_1,\dots,x_n)=y\). This implies that \[f(x_1,\dots,x_{i-1},g(x_1,\dots,x_{i-1},x_{i+1},\dots,x_n,y),x_{i+1},\dots,x_n)=y.\] Similarly, given \(x_1,\dots,x_n\in A\) let \(y= f(x_1,\dots,x_n)\). By definition \[g(x_1,\dots,x_{i-1},x_{i+1},\dots,x_n,y)\] is the unique member of \(A\) such that \[f(x_1,\dots,x_{i-1},g_i(x_1,\dots,x_{i-1},x_{i+1},\dots,x_n,y),x_{i+1},\dots,x_n)=y.\] Since we also have that \(f(x_1,\dots,x_n)=y\) it must be that \[g_i(x_1,\dots,x_{i-1},x_{i+1},\dots,x_n,y)=x_i\] and hence \[g_i(x_1,\dots,x_{i-1},x_{i+1},\dots,x_n,f(x_1,\dots,x_n))=x_i,\] as desired.
\end{proof}

We will speak of an \(n\)-quasigroup and its \(n\)-quasigroup magma reduct as though they are the same object. Traditionally a quasigroup is assumed to be binary, but we will default to our more expansive terminology and emphasize when we are talking about a binary quasigroup as opposed to an \(n\)-quasigroup for a general \(n\).

Since we can view \(n\)-quasigroups as magmas without losing any information we can take the variety of all \(n\)-quasigroups to be a subcategory of \(\magm_n\).

\begin{defn}[Category of quasigroups]
The \emph{category of \(n\)-quasigroups} \(\quasi_n\) is the full subcategory of \(\magm_n\) whose objects are \(n\)-quasigroups.
\end{defn}

\section{Alternating quasigroups}
\label{sec:alternating_quasigroups}
We are ready to name the class of quasigroups we will study.

\begin{defn}[Alternating quasigroup]
An \emph{alternating quasigroup} is a quasigroup which is also an alternating magma.
\end{defn}

Note that since the classes of alternating \(n\)-magmas and \(n\)-quasigroups are defined by universally quantified equations we have that the class of alternating \(n\)-quasigroups is equational and hence a variety of algebras. When \(n=2\) we have that \(\alt_2\) is trivial and hence the class of alternating binary quasigroups is the variety of binary quasigroups. Every group is thus an alternating binary quasigroup.

Since every variety of algebras forms a category whose morphisms are the usual homomorphisms of those algebras we have a category of alternating \(n\)-quasigroups for each \(n\in\PP\) whose structure doesn't depend on whether we consider our quasigroups as magmas or as algebras of signature \((n,\dots,n)\).

\begin{defn}[Category of alternating quasigroups]
The \emph{category of alternating \(n\)-quasigroups} \(\aq_n\) is the full subcategory of \(\magm_n\) (or of \(\quasi_n\)) whose objects are alternating \(n\)-quasigroups.
\end{defn}

Given any variety of algebras \(V\) (viewed as a category) and any set \(X\) there exists a free algebra \(\free_{V}(X)\), which is a member of \(V\), satisfying
    \[
        \hom_{V}(\free_{V}(X),A)\cong\hom_{\setcat}(X,A)
    \]
for any \(A\in V\). In particular, for any set \(X\) there exists a free alternating \(n\)-quasigroup \(\free_{\aq_n}(X)\) freely generated by \(X\). See section 4.3 of \cite{bergman} for more information.

Instead of looking at the whole category of alternating \(n\)-quasigroups, we restrict our attention to a subcategory whose morphism sets are restricted to those which preserve noncommutativity.

\begin{defn}[Commuting tuple]
Given \(A\in\aq_n\) we say that \(a\in A^n\) \emph{commutes} (or is a \emph{commuting tuple}) in \(A\) when we have for each \(\sigma\in\perm_n\) that
    \[
        f(a)=f(a_{\sigma(1)},\dots,a_{\sigma(n)}).
    \]
\end{defn}

\begin{defn}[Noncommuting tuple]
Given \(A\in\aq_n\) we say that \(a\in A^n\) \emph{does not commute} (or is a \emph{noncommuting tuple}) in \(A\) when \(a\) is not a commuting tuple in \(A\).
\end{defn}

The collection of all noncommuting tuples will be important for us.

\begin{defn}[Set of noncommuting tuples]
Given \(A\in\aq_n\) we define the \emph{noncommuting tuples} \(\nct(A)\) of \(A\) by
    \[
        \nct(A)=\set[a\in A^n]{a\text{ does not commute in }A}.
    \]
\end{defn}

The particular class of homomorphisms we will examine consists of those which preserve these sets of noncommuting tuples.

\begin{defn}[NC homomorphism]
We say that a homomorphism \(h\colon A_1\to A_2\) of alternating quasigroups is an \emph{NC homomorphism} (or a \emph{noncommuting homomorphism}) when for each \(a\in\nct(A_1)\) we have that
    \[
        h(a)=(h(a_1),\dots,h(a_n))\in\nct(A_2).
    \]
\end{defn}

Note that while embeddings are always NC homomorphisms the converse is not true. A typical NC homomorphism is not injective (or surjective, for that matter).

We can now define the category of quasigroups we will be using in our construction of manifolds.

\begin{defn}[Category of NC alternating quasigroups]
Given \(n\in\PP\) we define the \emph{category of  NC alternating \(n\)-quasigroups} to be the category \(\ncaq_n\) whose object class is \(\aq_n\) and whose morphisms are NC homomorphisms.
\end{defn}

We associate to each alternating \(n\)-quasigroup two sets of elements, which we will use when we produce pseudomanifolds from quasigroups.

\begin{defn}[Input elements]
Given a quasigroup \(A\) we define the \emph{input elements} of \(A\) to be
    \[
        \inp(A)=\set[x\in A]{(\exists a\in\nct(A))(\exists 1\le i\le n)(x=a_i)}.
    \]
\end{defn}

\begin{defn}[Output elements]
Given a quasigroup \(A\) we define the \emph{output elements} of \(A\) to be
    \[
        \out(A)=\set[f(a)]{a\in\nct(A)}.
    \]
\end{defn}

So far we have not investigated whether the quasigroups under consideration may be associative or not. Associative \(n\)-quasigroups are the \(n\)-ary groups introduced by Post \cite{post}. Certainly it is possible for a binary alternating quasigroup to be a group without being commutative, as any noncommutative group is an example. For higher arities it turns out that associativity implies commutativity.

\begin{prop}
Every alternating \(n\)-ary group for \(n\ge3\) is commutative.
\end{prop}

\begin{proof}
Suppose that \(A\) is an alternating \(n\)-ary group where \(n\ge3\) and let \(a_1,\dots,a_{n+1}\in A\). Observe that
    \begin{align*}
        f(f(a_1,\dots,a_n),a_{n-1},\dots,a_{n-1},a_{n+1}) &= \\ f(a_1,a_2,\dots,a_{n-2},f(a_{n-1},a_n,a_{n-1},\dots,a_{n-1}),a_{n+1}) &= \\
        f(a_1,a_2,\dots,a_{n-2},f(a_n,a_{n-1},a_{n-1},\dots,a_{n-1}),a_{n+1}) &= \\
        f(f(a_1,a_2,\dots,a_{n-2},a_n,a_{n-1}),a_{n-1},\dots,a_{n-1},a_{n+1}),
    \end{align*}
which implies that
    \[
        f(a_1,\dots,a_n)=f(a_1,\dots,a_{n-2},a_n,a_{n-1}),
    \]
so \(A\) is in fact commutative.
\end{proof}

\subsection{Examples of alternating quasigroups}
\label{subsec:examples_alternating}
We give three classes of examples of alternating quasigroups:
    \begin{enumerate}
        \item A particular order \(5\) ternary alternating quasigroup.
        \item Alternating \(n\)-quasigroups which are given as alternating products of commutative ones.
        \item Any binary quasigroup, which includes all groups.
    \end{enumerate}

After the authors laboriously produced (1) they prevailed upon Jonathan Smith for other examples of alternating \(n\)-quasigroups. Although no one had, to Smith's knowledge, studied the varieties of alternating \(n\)-quasigroups for \(n>2\) before he nonetheless provided a special case of (2), which we generalized. Other varieties of \(n\)-quasigroups have been investigated before, as discussed in the introduction as well as in \autoref{sec:latin_cubes}. Since binary quasigroups, and in particular groups, are well-known we won't discuss (3) any more in this section.

We are especially interested in alternating quasigroups which are not commutative. Although commutative quasigroups are an input for the alternating product construction in (2), we will focus on examples of noncommutative alternating quasigroups since these will yield nontrivial manifolds and appear more difficult to find.

\begin{example}
\label{ex:order_five}
Take \(S=(\Z/5\Z)^3\) and define \(h\colon\Z/5\Z\times\alt_3\to\perm_S\) by
    \[
        (h(k,\sigma))(x_1,x_2,x_3)=(x_{\sigma(1)}+k,x_{\sigma(2)}+k,x_{\sigma(3)}+k).
    \]
There are \(7\) members of \(\orb(h)\). One system of orbit representatives is:
    \[
        \set{000,011,022,012,021,013,031}
    \]
where we follow the convention that \(xyz\) indicates \((x,y,z)\) for \(x,y,z\in\Z/5\Z\). Let \(A=\Z/5\Z\) and define a ternary operation \(f\colon A^3\to A\) so that
    \[
        f((h(k,\sigma))(x_1,x_2,x_3))=f(x_1,x_2,x_3)+k
    \]
and \(f\) is defined on the above set of orbit representatives as follows.
\begin{center}
    \begin{tabular}{r|c}
        \(xyz\) & \(f(x,y,z)\) \\ \hline
        \(000\) & \(0\) \\
        \(011\) & \(0\) \\
        \(022\) & \(0\) \\
        \(012\) & \(3\) \\
        \(021\) & \(4\) \\
        \(013\) & \(4\) \\
        \(031\) & \(2\)
    \end{tabular}
\end{center}
One may verify that the above values and the symmetry imposed on \(f\) completely define a ternary alternating quasigroup operation on \(A\).
\end{example}

Since ternary alternating quasigroups are an equational class the above example yields infinitely many others. For example, if we take \(A\) to be the ternary alternating quasigroup in \autoref{ex:order_five} then we have that any algebra of the form \(\prod_{j\in J}A_j\) where \(A_j= A\) for all \(j\in J\) where \(J\) is any set is a ternary alternating quasigroup.

Our next class of examples makes use of a product-like construction which we now define.

\begin{defn}[Alternating map]
Given sets \(A\) and \(B\) we say that a function \(\alpha\colon A^n\to B\) is an \emph{\(n\)-ary alternating map} from \(A\) to \(B\) when for each \(\sigma\in\alt_n\) and each \(a\in A^n\) we have that
    \[
        \alpha(a)=\alpha(a_{\sigma(1)},\dots,a_{\sigma(n)}).
    \]
\end{defn}

The prototypical example of an alternating map is the determinant map
    \[
        \det\colon(\F^n)^n\to\F
    \]
where \(\F\) is any field. This is an \(n\)-ary alternating map from \(\F^n\) to \(\F\).

\begin{defn}[Alternating product]
Given an \(n\)-ary commutative quasigroup \((U,g)\), an \((n+1)\)-ary commutative quasigroup \((V,h)\), and an \(n\)-ary alternating map \(\alpha\colon A^n\to B\) the \emph{alternating product} of \(U\) and \(V\) with alternating map \(\alpha\) is the alternating \(n\)-quasigroup
    \[
        U\boxtimes_\alpha V=(U\times V,g\boxtimes_\alpha h\colon(U\times V)^n\to U\times V)
    \]
where for \((u_1,v_1),\dots,(u_n,v_n)\in U\times V\) we define
    \[
        (g\boxtimes_\alpha h)((u_1,v_1),\dots,(u_n,v_n))=(g(u),h(\alpha(u),v_1,\dots,v_n))
    \]
where \(u=(u_1,\dots,u_n)\).
\end{defn}

\begin{example}
\label{ex:field}
Let \(\F\) be a field of odd characteristic or characteristic \(0\) and let \(n\in\PP\). Define \(U=(\F^n,g)\) where
    \[
        g(u_1,\dots,u_n)=\sum_{i=1}^nu_i,
    \]
when \(u_1,\dots,u_n\in\F_q^n\), take \(V=(\F,h)\) where
    \[
        h(v_1,\dots,v_{n+1})=\sum_{i=1}^{n+1}v_i
    \]
when \(v_1,\dots,v_{n+1}\in\F\), and define \(\alpha=\det\colon(\F^n)^n\to\F\). The alternating \(n\)-quasigroup \(U\boxtimes_\alpha V\), which we will also denote by \(\F^{(n)}\), is the example Jonathan Smith indicated.
\end{example}

We could have allowed \(\F\) to have characteristic \(2\) as well, but the resulting alternating quasigroup \(\F^{(n)}\) would again be commutative, which will not give us an interesting object when we proceed to construct manifolds from quasigroups, although of course such commutative quasigroups may be used as the inputs in another alternating product.

The smallest \(n\)-quasigroup of the form of \autoref{ex:field} we can produce has \(\F=\F_3\) and hence has order \(3^{n+1}\). This is still slightly too large for us to work with ``by hand'', so we consider instead a more contrived example of smaller order.

\begin{example}
\label{ex:order_six}
Define \(U=(\Z/3\Z,g)\) where
    \[
        g(u_1,u_2,u_3)= u_1+u_2+u_3
    \]
for \(u_1,u_2,u_3\in\Z/3\Z\) and define \(V=(\Z/2\Z,h)\) where
    \[
        h(v_1,v_2,v_3,v_4)= v_1+v_2+v_3+v_4
    \]
for \(v_1,v_2,v_3,v_4\in\Z/2\Z\). Define an alternating map \(\alpha\colon(\Z/3\Z)^3\to\Z/2\Z\) by
    \[
        \alpha(x)=
            \begin{cases}
                1 &\text{when } x\in\orb_{\alt_3}(0,2,1) \\
                0 &\text{otherwise}.
            \end{cases}
    \]
We then have a ternary alternating quasigroup \(U\boxtimes_\alpha V\) of order \(6\).
\end{example}

Observe that our \autoref{ex:order_five} cannot be given as an alternating product since the order of an alternating product of two nontrivial commutative quasigroups must be composite and if at least one of the input quasigroups \(U\) or \(V\) is trivial then the resulting quasigroup will be commutative.

\section{Open serenation}
\label{sec:open_serenation}
We are now ready to begin our construction of manifolds from quasigroups. We have already described the functor \(\ogeo_n\colon\pmfld_n\to\topsp\). It remains to define a functor \(\simpcmplx_n\colon\ncaq_n\to\pmfld_n\) and to give an appropriate smooth atlas to \((\ogeo_n\circ\simpcmplx_n)(A)\) for each \(A\in\aq_n\). These tasks are performed in the subsections on simplicization and open serenation, respectively.

\subsection{Simplicization}
From each alternating quasigroup of arity \(n\) we obtain a pseudomanifold of dimension \(n\) whose facets may be oriented so that facets have compatible orientations where they meet at \((n-1)\)-faces. As motivation for our construction, consider a set \(A\) equipped with a binary operation \(f\colon A^2\to A\). Given elements \(a,b\in A\) we can represent that \(f(a,b)=c\) with a corresponding triangle as pictured in \autoref{fig:multiplication_traingle}. In this figure, the inputs \(a\) and \(b\) are decorated as \(\underline{a}\) and \(\underline{b}\), respectively, to indicate such. Similarly, the output \(c\) is decorated as \(\overline{c}\). The orientation of this triangle is given by the arc inside the triangle, which serves to indicate that \(a\) is the left input and \(b\) is the right input.

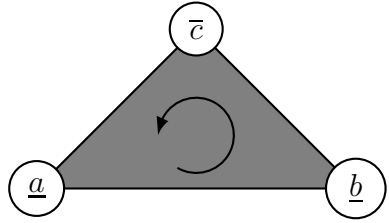
\begin{figure}
    \centering
    \begin{tikzpicture}[thick]
        \tikzstyle{every node}=[draw,circle,fill=white]
        \coordinate (a) at (5*360/8:3);
        \coordinate (b) at (7*360/8:3);
        \coordinate (c) at (0:0);
        \filldraw[fill=gray] (a) -- (b) -- (c) -- cycle;
        \draw[thick,-{Latex[bend]}] (barycentric cs:a=1,b=1,c=1)+(-120:0.5) arc(-120:180:0.5);
        \node at (a) {\(\underline{a}\)};
        \node at (b) {\(\underline{b}\)};
        \node at (c) {\(\overline{c}\)};
    \end{tikzpicture}
    \caption{Multiplication as a triangle}
    \label{fig:multiplication_traingle}
\end{figure}

If it also happens that \(d\in A\) with \(f(b,d)=c\) then we can continue our picture by adding another triangle as in \autoref{fig:another_triangle}. Note the compatible orientations of the two triangles in this figure.

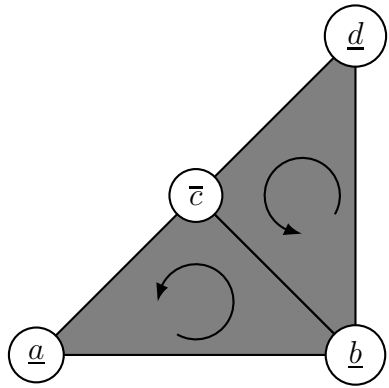
\begin{figure}
    \centering
    \begin{tikzpicture}[thick]
        \tikzstyle{every node}=[draw,circle,fill=white]
        \coordinate (a) at (5*360/8:3);
        \coordinate (b) at (7*360/8:3);
        \coordinate (c) at (0:0);
        \coordinate (d) at (1*360/8:3);
        \filldraw[fill=gray] (a) -- (b) -- (c) -- cycle;
        \draw[thick,-{Latex[bend]}] (barycentric cs:a=1,b=1,c=1)+(-120:0.5) arc(-120:180:0.5);
        \filldraw[fill=gray] (c) -- (b) -- (d) -- cycle;
        \draw[thick,-{Latex[bend]}] (barycentric cs:c=1,b=1,d=1)+(-30:0.5) arc(-30:270:0.5);
        \node at (a) {\(\underline{a}\)};
        \node at (b) {\(\underline{b}\)};
        \node at (c) {\(\overline{c}\)};
        \node at (d) {\(\underline{d}\)};
    \end{tikzpicture}
    \caption{Adding another triangle}
    \label{fig:another_triangle}
\end{figure}

We may continue in this fashion, building a simplicial complex whose vertices are \(\underline{x}\) and \(\overline{x}\) for \(x\in A\) and whose facets are of the form \(\set{\underline{x},\underline{y},\overline{f(x,y)}}\). See \autoref{fig:more_triangles} for a slightly larger example.

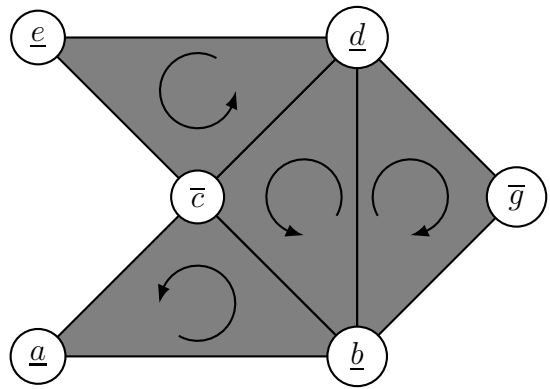
\begin{figure}
    \centering
    \begin{tikzpicture}[thick]
        \tikzstyle{every node}=[draw,circle,fill=white]
        \coordinate (a) at (5*360/8:3);
        \coordinate (b) at (7*360/8:3);
        \coordinate (c) at (0:0);
        \coordinate (d) at (1*360/8:3);
        \coordinate (e) at (3*360/8:3);
        \coordinate (g) at (0:4.24);
        \filldraw[fill=gray] (a) -- (b) -- (c) -- cycle;
        \draw[thick,-{Latex[bend]}] (barycentric cs:a=1,b=1,c=1)+(-120:0.5) arc(-120:180:0.5);
        \filldraw[fill=gray] (c) -- (b) -- (d) -- cycle;
        \draw[thick,-{Latex[bend]}] (barycentric cs:c=1,b=1,d=1)+(-30:0.5) arc(-30:270:0.5);
        \filldraw[fill=gray] (c) -- (e) -- (d) -- cycle;
        \draw[thick,-{Latex[bend]}] (barycentric cs:c=1,e=1,d=1)+(60:0.5) arc(60:360:0.5);
        \filldraw[fill=gray] (g) -- (b) -- (d) -- cycle;
        \draw[thick,-{Latex[bend]}] (barycentric cs:g=1,b=1,d=1)+(210:0.5) arc(210:-90:0.5);
        \node at (a) {\(\underline{a}\)};
        \node at (b) {\(\underline{b}\)};
        \node at (c) {\(\overline{c}\)};
        \node at (d) {\(\underline{d}\)};
        \node at (e) {\(\underline{e}\)};
        \node at (g) {\(\overline{g}\)};
    \end{tikzpicture}
    \caption{Adding more triangles}
    \label{fig:more_triangles}
\end{figure}

If it happens that \(f(a,b)=f(b,a)\) then we will have two triangles with the same vertices. Our solution is to only form triangles with vertex sets \(\set{\underline{a},\underline{b},\overline{f(a,b)}}\) when \(a\) and \(b\) do not commute under \(f\). In \autoref{fig:commuting_triangles} we see the undesirable case where
    \[
        f(a,b)=c=f(b,a)
    \]
Note the incompatible orientations along the edge joining \(\underline{a}\) and \(\underline{b}\).

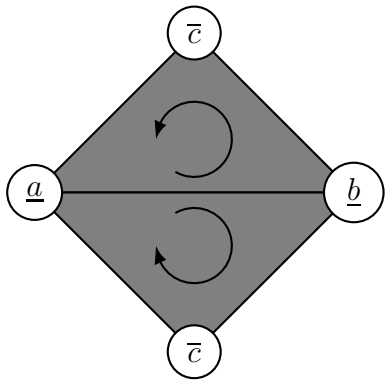
\begin{figure}
    \centering
    \begin{tikzpicture}[thick]
        \tikzstyle{every node}=[draw,circle,fill=white]
        \coordinate (a) at (5*360/8:3);
        \coordinate (b) at (7*360/8:3);
        \coordinate (c) at (0:0);
        \coordinate (c') at (3*360/4:4.24);
        \filldraw[fill=gray] (a) -- (b) -- (c) -- cycle;
        \draw[thick,-{Latex[bend]}] (barycentric cs:a=1,b=1,c=1)+(-120:0.5) arc(-120:180:0.5);
        \filldraw[fill=gray] (a) -- (b) -- (c') -- cycle;
        \draw[thick,-{Latex[bend]}] (barycentric cs:b=1,a=1,c'=1)+(120:0.5) arc(120:-180:0.5);
        \node at (a) {\(\underline{a}\)};
        \node at (b) {\(\underline{b}\)};
        \node at (c) {\(\overline{c}\)};
        \node at (c') {\(\overline{c}\)};
    \end{tikzpicture}
    \caption{Triangles where \(f(a,b)=f(b,a)\)}
    \label{fig:commuting_triangles}
\end{figure}

We can generalize this situation to the creation of a \(n\)-dimensional simplicial complex from an \(n\)-ary operation \(f\colon A^n\to A\). Suppose now that \(f\colon A^3\to A\) is a ternary operation. Given elements \(a,b,c,d\in A\) we can represent that \(f(a,b,c)=d\) with a corresponding oriented tetrahedron as in \autoref{fig:multiplication_tetrahedron}. The orientation in this case would be \((\underline{a},\underline{b},\underline{c},\overline{d})/\alt_4\).

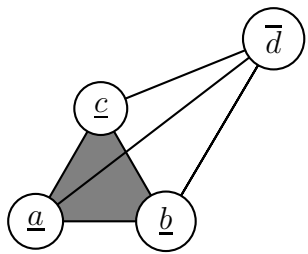
\begin{figure}
    \centering
    \begin{tikzpicture}[thick]
        \tikzstyle{every node}=[draw,circle,fill=white]
        \coordinate (a) at (210:1);
        \coordinate (b) at (330:1);
        \coordinate (c) at (90:1);
        \coordinate (d) at (2*180/9:3);
        \filldraw[fill=gray] (a) -- (b) -- (c) -- cycle;
        \draw (a) -- (d) -- (b) -- (d) -- (c);
        \node at (a) {\(\underline{a}\)};
        \node at (b) {\(\underline{b}\)};
        \node at (c) {\(\underline{c}\)};
        \node at (d) {\(\overline{d}\)};
    \end{tikzpicture}
    \caption{Ternary multiplication as a tetrahedron}
    \label{fig:multiplication_tetrahedron}
\end{figure}

We now have a different problem if we would like to obtain a pseudomanifold, which is that up to six tetrahedra could meet at the triangle with vertex set \(\set{\underline{a},\underline{b},\underline{c}}\). This situation is pictured in \autoref{fig:too_many_tetrahedra}.

\begin{figure}
    \centering
    \begin{tikzpicture}[thick]
        \tikzstyle{every node}=[draw,fill=white]
        \coordinate (a) at (210:1);
        \coordinate (b) at (330:1);
        \coordinate (c) at (90:1);
        \coordinate (abc) at (1*180/9:3);
        \coordinate (bca) at (2*180/9:3);
        \coordinate (cab) at (7*180/9:3);
        \coordinate (bac) at (8*180/9:3);
        \coordinate (acb) at (12.5*180/9:3);
        \coordinate (cba) at (14.5*180/9:3);
        \filldraw[fill=gray] (a) -- (b) -- (c) -- cycle;
        \draw (a) -- (abc) -- (b) -- (abc) -- (c);
        \draw (a) -- (bca) -- (b) -- (bca) -- (c);
        \draw (a) -- (cab) -- (b) -- (cab) -- (c);
        \draw (a) -- (bac) -- (b) -- (bac) -- (c);
        \draw (a) -- (acb) -- (b) -- (acb) -- (c);
        \draw (a) -- (cba) -- (b) -- (cba) -- (c);
        \node[circle] at (a) {\(\underline{a}\)};
        \node[circle] at (b) {\(\underline{b}\)};
        \node[circle] at (c) {\(\underline{c}\)};
        \node at (abc) {\(\overline{f(a,b,c)}\)};
        \node at (bca) {\(\overline{f(b,c,a)}\)};
        \node at (cab) {\(\overline{f(c,a,b)}\)};
        \node at (bac) {\(\overline{f(b,a,c)}\)};
        \node at (acb) {\(\overline{f(a,c,b)}\)};
        \node at (cba) {\(\overline{f(c,b,a)}\)};
    \end{tikzpicture}
    \caption{Too many tetrahedra meeting}
    \label{fig:too_many_tetrahedra}
\end{figure}
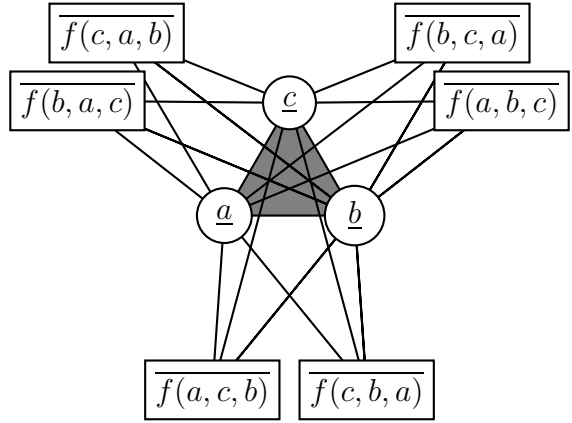

Our solution is to require that \(f\) is invariant under even permutations of its arguments. In this case, \(f(a,b,c)=f(b,c,a)=f(c,a,b)\) but in general \(f(a,b,c)\neq f(b,a,c)\). In \autoref{fig:alternating_tetrahedra} we see how this returns us to satisfying the pseudomanifold condition that a unique pair of facets meet at each \((n-1)\)-face.

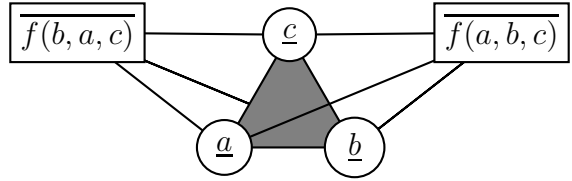
\begin{figure}
    \centering
    \begin{tikzpicture}[thick]
        \tikzstyle{every node}=[draw,fill=white]
        \coordinate (a) at (210:1);
        \coordinate (b) at (330:1);
        \coordinate (c) at (90:1);
        \coordinate (abc) at (1*180/9:3);
        \coordinate (bac) at (8*180/9:3);
        \draw (a) -- (bac) -- (b) -- (bac) -- (c);
        \filldraw[fill=gray] (a) -- (b) -- (c) -- cycle;
        \draw (a) -- (abc) -- (b) -- (abc) -- (c);
        \node[circle] at (a) {\(\underline{a}\)};
        \node[circle] at (b) {\(\underline{b}\)};
        \node[circle] at (c) {\(\underline{c}\)};
        \node at (abc) {\(\overline{f(a,b,c)}\)};
        \node at (bac) {\(\overline{f(b,a,c)}\)};
    \end{tikzpicture}
    \caption{Exactly two tetrahedra meeting}
    \label{fig:alternating_tetrahedra}
\end{figure}

\begin{defn}[Simplicization functor]
We define a \emph{simplicization functor}
    \[
        \simpcmplx_n\colon\ncaq_n\to\pmfld_n
    \]
as follows. Given \(A\in\aq_n\) we define
    \[
        \simpcmplx_n(A)=\bigcup_{\mathclap{a\in\nct(A)}}\pow\left(\set{\underline{a}_1,\dots,\underline{a}_n,\overline{f(a)}}\right)
    \]
where \(\underline{a}\) indicates an element \(a\in A\) we view as coming from \(\inp(A)\), \(\overline{a}\) indicates an element \(a\in A\) we view as coming from \(\out(A)\), and \(\underline{a}\neq\overline{a}\) so that
    \[
        \vertices(\simpcmplx_n(A))=\set[\underline{a}]{a\in\inp(A)}\cup\set[\overline{a}]{a\in\out(A)}
    \]
is a disjoint union of \(\inp(A)\) and \(\out(A)\). Given an NC homomorphism \(h\colon A_1\to A_2\) we define
    \[
        \simpcmplx_n(h)\colon\simpcmplx_n(A_1)\to\simpcmplx_n(A_2)
    \]
by
    \[
        \simpcmplx_n(h)(\underline{a})=\underline{h(a)}
    \]
and
    \[
        \simpcmplx_n(h)(\overline{a})=\overline{h(a)}.
    \]
\end{defn}

We will begin to drop the subscript \(n\) going forward unless we need to explicitly refer to the arity or dimension under consideration. We refer to \(\simpcmplx(A)\) as the \emph{simplicization} of \(A\) and use similar language for the simplicial maps \(\simpcmplx(h)\).

\subsection{Open serenation}
Finally we can describe our functor taking \(\ncaq_n\) to \(\smfld_n\). We give the relevant coordinate charts first. The domain for all of our coordinate charts is a particular bipyramid situated on the origin and standard basis points in \(\R^n\).

\begin{defn}[Bipyramid]
The \emph{standard open bipyramid} (or just \emph{bipyramid}) in \(\R^n\) is
    \[
        \bipyr_n=\ocvx\left(\set{(0,\dots,0),\left(\frac{2}{n},\dots,\frac{2}{n}\right)}\cup\set{e_1,\dots,e_n}\right)
    \]
where \(e_i\) is the \(i^{\text{th}}\) standard basis vector of \(\R^n\).
\end{defn}

The \(n=3\) case of a bipyramid is illustrated in \autoref{fig:bipyramid}.

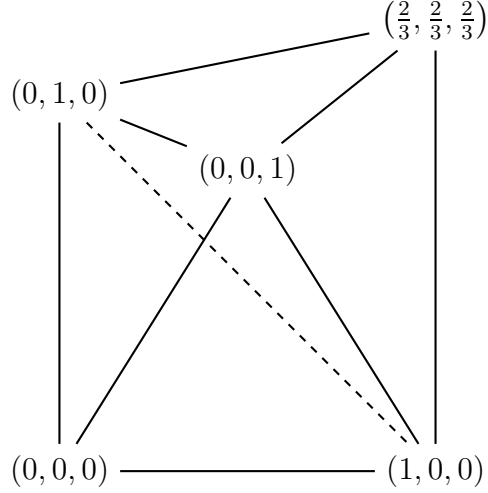
\begin{figure}
    \centering
    \begin{tikzpicture}[thick,scale=5]
        \node (a) at (0,0) {\((0,0,0)\)};
        \node (b) at (1,0) {\((1,0,0)\)};
        \node (c) at (0,1) {\((0,1,0)\)};
        \node (d) at (1/2,4/5) {\((0,0,1)\)};
        \node (e) at (1,6/5) {\(\left(\frac{2}{3},\frac{2}{3},\frac{2}{3}\right)\)};
        \draw (a) -- (b) -- (d) -- (a) -- (c) -- (e) -- (d) -- (c);
        \draw (b) -- (e);
        \draw[dashed] (c) -- (b);
    \end{tikzpicture}
    \caption{A bipyramid}
    \label{fig:bipyramid}
\end{figure}

We have two types of charts, but they're quite similar to each other. Given \(a\in\nct(A)\) let \(a'\in\nct(A)\) be a permutation of \(a\) obtained by swapping two entries.

\begin{defn}[Serene chart of input type]
Given an alternating \(n\)-quasigroup \(A\) and \(a=(a_1,\dots,a_n)\in\nct(A)\) the \emph{serene chart of input type} for \(a\) is
    \[
        \underline{\phi}_a\colon\bipyr_n\to\ogeo_n(\simpcmplx_n(A))
    \]
where we set
    \[
        \underline{\phi}_a(u_1,\dots,u_n)=\sum_{i=1}^nu_i\underline{a}_i+\left(1-\sum_{i=1}^nu_i\right)\overline{f(a)}
    \]
when \(\sum_{i=1}^nu_i\le1\) and
    \begin{align*}
        \underline{\phi}_a(u_1,\dots,u_n)=\frac{2}{n}\sum_{i=1}^n\left(1+\frac{n-2}{2}u_i-\sum_{j\neq i}u_j\right)\underline{a}_i+\left(-1+\sum_{i=1}^nu_i\right)\overline{f(a')}
    \end{align*}
when \(\sum_{i=1}^nu_i>1\).
\end{defn}

The image of the vertices of the bipyramid under the map \(\underline{\phi}_a\) for the case \(n=3\) is illustrated in \autoref{fig:chart_vertex_images}. Note that strictly speaking these vertices are not in the domain of \(\underline{\phi}_a\), but we can extend the formula defining \(\underline{\phi}_a\) to them in order to see where the boundaries of the image of \(\bipyr_3\) under this map lie. In \autoref{fig:chart_generic} the image of the \(n\)-dimensional bipyramid is depicted and labeled with the two cases of the formula for \(\underline{\phi}_a\). In this figure, the triangle with vertices \(\underline{a}_1\), \(\underline{a}_2\), and \(\overline{f(a)}\) represents the \(n\)-simplex which is the image of the origin and the standard basis vectors \(e_i\) while the triangle with vertices \(\underline{a}_1\), \(\underline{a}_2\), and \(\overline{f(a')}\) represents the \(n\)-simplex which is the image of the point \(\left(\frac{2}{n},\dots,\frac{2}{n}\right)\) and the standard basis vectors \(e_i\). The line joining \(\underline{a}_1\) and \(\underline{a}_2\) represents the \((n-1)\)-simplex determined by the standard basis vectors.

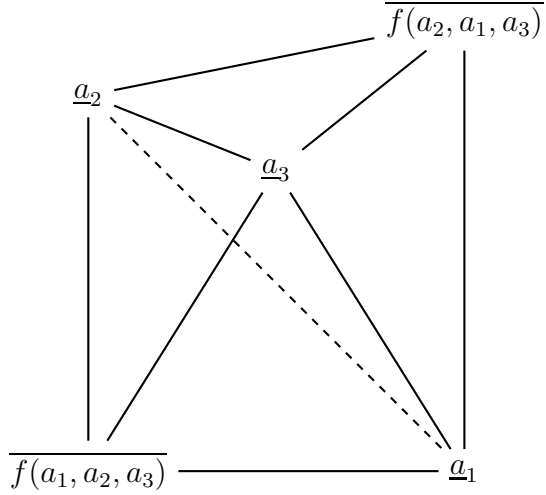
\begin{figure}
    \centering
    \begin{tikzpicture}[thick,scale=5]
        \node (a) at (0,0) {\(\overline{f(a_1,a_2,a_3)}\)};
        \node (b) at (1,0) {\(\underline{a}_1\)};
        \node (c) at (0,1) {\(\underline{a}_2\)};
        \node (d) at (1/2,4/5) {\(\underline{a}_3\)};
        \node (e) at (1,6/5) {\(\overline{f(a_2,a_1,a_3)}\)};
        \draw (a) -- (b) -- (d) -- (a) -- (c) -- (e) -- (d) -- (c);
        \draw (b) -- (e);
        \draw[dashed] (c) -- (b);
    \end{tikzpicture}
    \caption{Images of the vertices of \(\bipyr_3\) under \(\underline{\phi}_a\)}
    \label{fig:chart_vertex_images}
\end{figure}

\begin{figure}
    \centering
    \begin{tikzpicture}[thick,scale=0.8]
        \node (fa) at (0,0) {\(\overline{f(a)}\)};
        \node (a1) at (10,0) {\(\underline{a}_1\)};
        \node (a2) at (0,10) {\(\underline{a}_2\)};
        \node (fb) at (10,10) {\(\overline{f(a')}\)};
        \node at (3.7,2.3) {\(\sum_{i=1}^nu_i\underline{a}_i+\left(1-\sum_{i=1}^nu_i\right)\overline{f(a)}\)};
        \node at (6,8.2) {\(\frac{2}{n}\sum_{i=1}^n\left(1+\frac{n-2}{2}u_i-\sum_{j\neq i}u_j\right)\underline{a}_i+\)};
        \node at (7,7) {\(\left(-1+\sum_{i=1}^nu_i\right)\overline{f(a')}\)};
        \draw (fa) -- (a1) -- (fb) -- (a2) -- (fa);
        \draw (a1) -- (a2);
    \end{tikzpicture}
    \caption{Cases of \(\underline{\phi}_a\) for a generic \(\bipyr_n\)}
    \label{fig:chart_generic}
\end{figure}
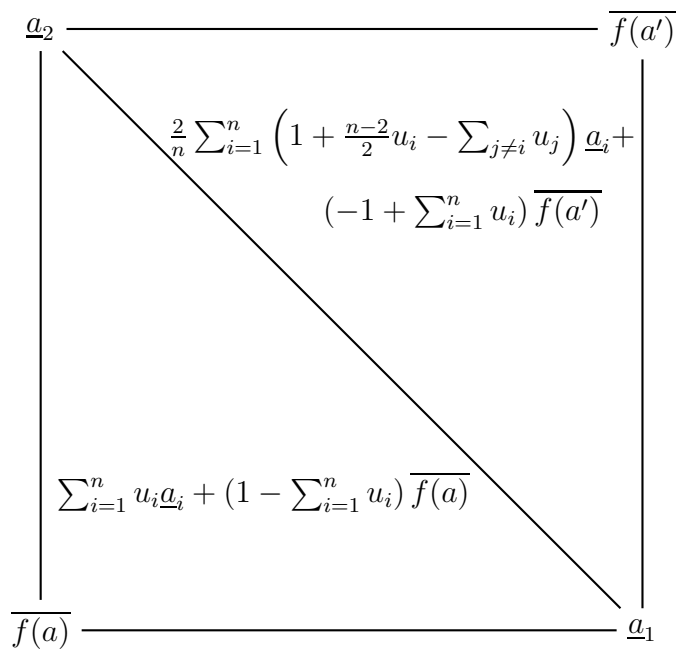

The reader may find that the formula for \(\underline{\phi}_a(u_1,\dots,u_n)\) when \(\sum_{i=1}^nu_i>1\) is not entirely obvious. We sketch its derivation, which is visualized in \autoref{fig:chart_calculation}.

\begin{figure}
    \centering
    \begin{tikzpicture}[thick,scale=0.8]
        \node (fa) at (0,0) {\(\overline{f(a)}\)};
        \node (a1) at (10,0) {\(\underline{a}_1\)};
        \node (a2) at (0,10) {\(\underline{a}_2\)};
        \node (fb) at (10,10) {\(\overline{f(a')}\)};
        \node (m) at (5,5) {\(\left(\frac{1}{n},\dots,\frac{1}{n}\right)\)};
        \node (w) at (4,6) {\(w\)};
        \node (v) at (2,4) {\(v\)};
        \node (u) at (6,8) {\(u\)};
        \draw (fa) -- (a1) -- (fb) -- (a2) -- (fa);
        \draw (a1) -- (m) -- (w) -- (a2);
        \draw[->] (fa) -- (m);
        \draw[->] (v) -- (w);
        \draw[->] (w) -- (u);
    \end{tikzpicture}
    \caption{Deriving the formula for a chart of input type}
    \label{fig:chart_calculation}
\end{figure}
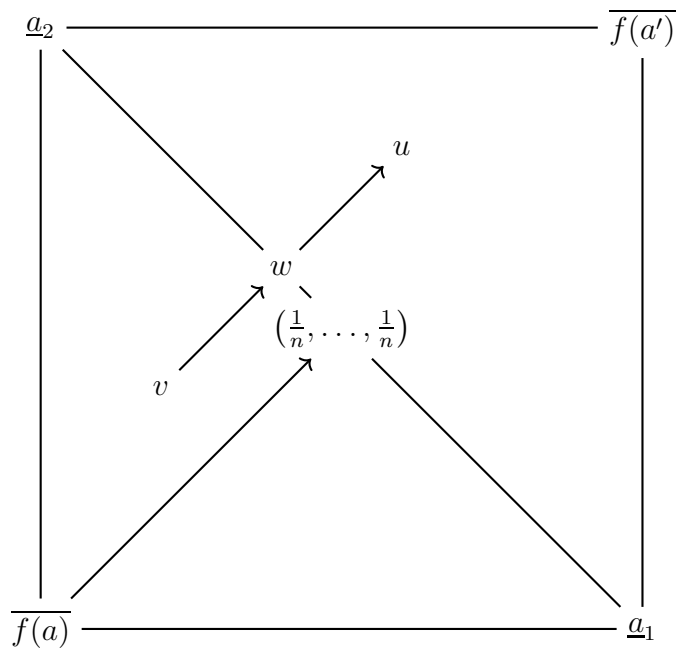

A point \(u\in\bipyr_n\) either has 
    \begin{enumerate}
        \item \(\sum_{i=1}^nu_i<1\), in which case it is mapped to a point in the open convex hull of the \(\underline{a}_i\) and \(\overline{f(a)}\),
        \item \(\sum_{i=1}^nu_i=1\), in which case it is mapped to a point in the open convex hull of the \(\underline{a}_i\), or
        \item \(\sum_{i=1}^nu_i>1\), in which case it is mapped to a point in the open convex hull of the \(\underline{a}_i\) and \(\overline{f(a')}\).
    \end{enumerate}
It is this last case which we need to consider carefully. Note that
    \[
        \underline{\phi}_a^{-1}\left(\ocvx\left(\set{\underline{a}_1,\dots,\underline{a}_n,\overline{f(a')}}\right)\right)
    \]
is the reflection of
    \[
        \underline{\phi}_a^{-1}\left(\ocvx\left(\set{\underline{a}_1,\dots,\underline{a}_n,\overline{f(a)}}\right)\right)
    \]
over the affine hyperplane through the point \(\left(\frac{1}{n},\dots,\frac{1}{n}\right)\) orthogonal to the vector \(\left(\frac{1}{n},\dots,\frac{1}{n}\right)\). A point \(u\in\bipyr_n\) with \(\sum_{i=1}^nu_i>1\) has a mirror image \(v\in\bipyr_n\) with \(\sum_{i=1}^nv_i<1\) where
    \[
        u=v+\gamma\left(\frac{1}{n},\dots,\frac{1}{n}\right)
    \]
for some \(\gamma\in\R\).
Since we set
    \[
        \underline{\phi}_a(v)=\sum_{i=1}^nv_i\underline{a}_i+\left(1-\sum_{i=1}^nv_i\right)\overline{f(a)}
    \]
we should analogously set
    \[
        \underline{\phi}_a(u)=\sum_{i=1}^nv_i\underline{a}_i+\left(1-\sum_{i=1}^nv_i\right)\overline{f(a')}.
    \]

We would like to find a formula for the \(v_i\) in terms of the \(u_i\). We accomplish this by finding the point \(w\in\ocvx(\set{e_1,\dots,e_n})\) which also lies on the line determined by \(u\) and \(v\). We can then write \(v=u-2(u-w)\). For any \(i\neq n\) we have that
    \[
        (u-w)\cdot(e_i-e_n)=0
    \]
so \(w_i-w_n=u_i-u_n\). We also know that \(\sum_{i=1}^nw_i=1\) so taking the \(u_i\) as constants we have a system of \(n\) linear equations in the \(n\) unknowns \(w_1,\dots,w_n\). Solving this system yields that
    \[
        w_i=\frac{1}{n}\left(1+(n-1)u_i-\sum_{j\neq i}u_j\right).
    \]
Since \(v=u-2(u-w)=2w-u\) we obtain
    \[
        v_i=\frac{2}{n}\left(1+(n-1)u_i-\sum_{j\neq i}u_j\right)-u_i=\frac{2}{n}\left(1+\frac{n-2}{2}u_i-\sum_{j\neq i}u_j\right).
    \]
The coefficient of \(\overline{f(a')}\) in \(\underline{\phi}_a(u)\) is then
    \begin{align*}
        1-\sum_{i=1}^nv_i &= 1-\frac{2}{n}\sum_{i=1}^n\left(1+\frac{n-2}{2}u_i-\sum_{j\neq i}u_j\right) \\
        &= -1+\sum_{i=1}^nu_i.
    \end{align*}

We produce another family of charts. Given \(a\in\nct(A)\) let \(a_{n+1}\in A\) be the unique solution to \(f(a)=f(b)\) where
    \[
        b_i=
        \begin{cases}
            a_i &\text{when } 1\le i\le n-2 \\
            a_{n+1} &\text{when } i=n-1 \\
            a_{n-1} &\text{when } i=n
        \end{cases}
        .
    \]

\begin{defn}[Serene chart of output type]
Given an alternating \(n\)-quasigroup \(A\) and \(a=(a_1,\dots,a_n)\in\nct(A)\) the \emph{serene chart of output type} for \(a\) is
    \[
        \overline{\phi}_a\colon\bipyr_n\to\ogeo_n(\simpcmplx_n(A))
    \]
where we set
    \[
        \overline{\phi}_a(u_1,\dots,u_n)=\sum_{i=1}^{n-1}u_i\underline{a}_i+u_n\overline{f(a)}+\left(1-\sum_{i=1}^nu_i\right)\underline{a}_n
    \]
when \(\sum_{i=1}^nu_i\le1\) and
    \begin{align*}
        \overline{\phi}_a(u_1,\dots,u_n)=\frac{2}{n}\sum_{i=1}^{n-1}\left(1+\frac{n-2}{2}u_i-\sum_{j\neq i}u_j\right)\underline{a}_i\\+\frac{2}{n}\left(1+\frac{n-2}{2}u_n-\sum_{j=1}^{n-1}u_j\right)\overline{f(a)}\\+\left(-1+\sum_{i=1}^nu_i\right)\underline{a}_{n+1}
    \end{align*}
when \(\sum_{i=1}^nu_i>1\).
\end{defn}

\begin{defn}[Open serenation functor]
We define an \emph{open serenation functor}
    \[
        \oser_n\colon\ncaq_n\to\smfld_n
    \]
as follows. Given \(A\in\aq_n\) we define
    \[
        \oser_n(A)=(\ogeo_n\circ\simpcmplx_n)(A)
    \]
where the smooth the atlas of \(\oser_n(A)\) is
    \[
        \bigcup_{\mathclap{a\in\nct(A)}}\set{\underline{\phi}_a,\overline{\phi}_a}.
    \]
Given an NC homomorphism \(h\colon A\to B\) where \(A,B\in\aq_n\) we define
    \[
        \oser_n(h)=(\ogeo_n\circ\simpcmplx_n)(h).
    \]
\end{defn}

As in the case of the simplicization functor we'll generally omit the dimension subscript unless it's germane to the discussion at hand.

The tangent spaces of \(\oser(A)\) can be described directly in terms of the quasigroup \(A\). Given a point \(x\) in a manifold we denote by \(T_x\) the set of tangent vectors at \(x\) and by \(T_x\) the tangent space at \(x\).

\begin{prop}
Given a quasigroup \(A\), \(a\in\nct(A)\), and \(x\in\oser(A)\) with
    \[
        x\in\ocvx\left(\set{\underline{a}_1,\dots,\underline{a}_n,\overline{f(a)}}\right)
    \]
we have that
    \[
        T_x=\spn\left(\set[\underline{a}_i-\overline{f(a)}]{i\in[n]}\right).
    \]
\end{prop}

\begin{proof}
We obtain the same tangent space with any chart so take \(a=(a_1,\dots,a_n)\in\nct(A)\) and consider that by our assumption
    \[
        \underline{\phi}_a^{-1}(x)=(x_1,\dots,x_n)
    \]
where \(\sum_{i=1}^nx_i<1\). It follows that near \(\underline{\phi}_a^{-1}(x)\) the map \(\underline{\phi}_a\) is given by
    \[
        \underline{\phi}_a(u_1,\dots,u_n)=\sum_{i=1}^nu_i\underline{a}_i+\left(1-\sum_{i=1}^nu_i\right)\overline{f(a)}.
    \]
Note that \(\underline{\phi}_a\) is a smooth map from \(\R^n\) to \(\R^{\vertices(\simpcmplx_n(A))}\) near \(\underline{\phi}_a^{-1}(x)\). For each \(i\in[n]\) we have a tangent vector of the form
    \[
        \pdif{}{u_i}\underline{\phi}_a=\underline{a}_i-\overline{f(a)},
    \]
as claimed.
\end{proof}

The points addressed by the previous proposition cover almost all of \(\oser(A)\). We have a similar description of the tangent space to a point on one of the remaining ``creases'', of which there are two types: those which are the open convex hull of an \((n-1)\)-face containing an output vertex and those which are the open convex hull of an \((n-1)\)-face containing only input vertices.

\begin{prop}
Given a quasigroup \(A\), \(a\in\nct(A)\), and \(x\in\oser(A)\) with
    \[
        x\in\ocvx\left(\set{\underline{a}_1,\dots,\underline{a}_{n-1},\overline{f(a)}}\right)
    \]
we have that
    \[
        T_x=\spn\left(\set{\underline{a}_1,\dots,\underline{a}_{n-1},\overline{f(a)}}\right).
    \]
\end{prop}

\begin{proof}
Fix some \(k\in[n]\) and let \(\epsilon>0\) be sufficiently small that
    \[
        x+te_k\in\bipyr_n
    \]
for all \(t\in(-\epsilon,\epsilon)\). Let \(\gamma\colon(-\epsilon,\epsilon)\to\oser_n(A)\) be given by
    \[
        \gamma(t)=\underline{\phi}_a(x+te_k).
    \]
We have a corresponding tangent vector
    \[
        \left.\diff{(\underline{\phi}_a^{-1}\circ\gamma)}{t}\right|_{t=0}=e_k,
    \]
which we naturally identify with \(\underline{a}_k\) when \(k<n\) and with \(\overline{f(a)}\) when \(k=n\).
\end{proof}

\begin{prop}
Given a quasigroup \(A\), \(a\in\nct(A)\), and \(x\in\oser(A)\) with
    \[
        x\in\ocvx\left(\set{\underline{a}_1,\dots,\underline{a}_n}\right)
    \]
we have that
    \[
        T_x=\spn\left(\set{\underline{a}_1,\dots,\underline{a}_n}\right).
    \]
\end{prop}

\begin{proof}
The argument here is identical to that for the preceding proposition with the label \(\overline{f(a)}\) replaced with \(\underline{a}_n\) and the chart \(\underline{\phi}_a\) is replaced with \(\overline{\phi}_a\).
\end{proof}

\subsection{Examples of open serenation}
\label{subsec:examples_open_serenation}
We give some small examples of open serenation. Our first illustrates the distinction between open serenation and the construction for groups described by Herman and Pakianathan \cite{herman}.

\begin{example}
\label{ex:quaternion_open}
Let \(G=\set{\pm1,\pm i,\pm j,\pm k}\) be the quaternion group of order \(8\). We have that
    \[
        \nct(G)=\set[(\pm u,\pm v)]{\set{u,v}\in\binom{\set{i,j,k}}{2}}
    \]
so
    \[
        \inp(G)=\set{\pm i,\pm j,\pm k}
    \]
and
    \[
        \out(G)=\set{\pm i,\pm j,\pm k}.
    \]
We see that
    \[
        \vertices(\simpcmplx(G))=\set[\underline{\pm u}]{u\in\set{i,j,k}}\cup\set[\overline{\pm u}]{u\in\set{i,j,k}}.
    \]
\autoref{fig:quaternion_facets} gives the facet of \(\simpcmplx(G)\) associated with each \((a,b)\in\nct(G)\).

\begin{table}
    \centering
    \begin{tabular}{c|c}
            \((a,b)\in\nct(G)\) & \(\sigma\in\simpcmplx(G)\) \\ \hline
            \((i,j)\) & \(\set{\underline{i},\underline{j},\overline{k}}\) \\
            \((i,-j)\) & \(\set{\underline{i},\underline{-j},\overline{-k}}\) \\
            \((-i,j)\) & \(\set{\underline{-i},\underline{j},\overline{-k}}\) \\
            \((-i,-j)\) & \(\set{\underline{-i},\underline{-j},\overline{k}}\) \\
            \((i,k)\) & \(\set{\underline{i},\underline{k},\overline{-j}}\) \\
            \((i,-k)\) & \(\set{\underline{i},\underline{-k},\overline{j}}\) \\
            \((-i,k)\) & \(\set{\underline{-i},\underline{k},\overline{j}}\) \\
            \((-i,-k)\) & \(\set{\underline{-i},\underline{-k},\overline{-j}}\) \\
            \((j,i)\) & \(\set{\underline{j},\underline{i},\overline{-k}}\) \\
            \((j,-i)\) & \(\set{\underline{j},\underline{-i},\overline{k}}\) \\
            \((-j,i)\) & \(\set{\underline{-j},\underline{i},\overline{k}}\) \\
            \((-j,-i)\) & \(\set{\underline{-j},\underline{-i},\overline{-k}}\) \\
            \((j,k)\) & \(\set{\underline{j},\underline{k},\overline{i}}\) \\
            \((j,-k)\) & \(\set{\underline{j},\underline{-k},\overline{-i}}\) \\
            \((-j,k)\) & \(\set{\underline{-j},\underline{k},\overline{-i}}\) \\
            \((-j,-k)\) & \(\set{\underline{-j},\underline{-k},\overline{i}}\) \\
            \((k,i)\) & \(\set{\underline{k},\underline{i},\overline{j}}\) \\
            \((k,-i)\) & \(\set{\underline{k},\underline{-i},\overline{-j}}\) \\
            \((-k,i)\) & \(\set{\underline{-k},\underline{i},\overline{-j}}\) \\
            \((-k,-i)\) & \(\set{\underline{-k},\underline{-i},\overline{j}}\) \\
            \((k,j)\) & \(\set{\underline{k},\underline{j},\overline{-i}}\) \\
            \((k,-j)\) & \(\set{\underline{k},\underline{-j},\overline{i}}\) \\
            \((-k,j)\) & \(\set{\underline{-k},\underline{j},\overline{i}}\) \\
            \((-k,-j)\) & \(\set{\underline{-k},\underline{-j},\overline{-i}}\)
        \end{tabular}
        
\medskip
\medskip
    \caption{Facets of \(\simpcmplx(G)\)}
    \label{fig:quaternion_facets}
\end{table}

The induced subcomplex on \(\set[\underline{\pm u}]{u\in\set{i,j,k}}\) is the graph pictured in \autoref{fig:quaternion_subcomplex}.
    \begin{figure}
        \centering
            \begin{tikzpicture}[thick]
                \tikzstyle{every node}=[draw,circle]
                \node (i) at (0:3) {\(\underline{i}\)};
                \node (j) at (360/6:3) {\(\underline{j}\)};
                \node (k) at (2*360/6:3) {\(\underline{k}\)};
                \node (-i) at (3*360/6:3) {\(\underline{-i}\)};
                \node (-j) at (4*360/6:3) {\(\underline{-j}\)};
                \node (-k) at (5*360/6:3) {\(\underline{-k}\)};
                \draw (i) -- (j) -- (-i) -- (-j) -- (i);
                \draw (i) -- (k) -- (-i) -- (-k) -- (i);
                \draw (j) -- (k) -- (-j) -- (-k) -- (j);
            \end{tikzpicture}
        \caption{A subcomplex of \(\simpcmplx(G)\)}
        \label{fig:quaternion_subcomplex}
    \end{figure}
We may decompose this graph into three \(4\)-cycles, which are
    \[
        (\underline{i},\underline{j},\underline{-i},\underline{-j})\text{, }(\underline{i},\underline{k},\underline{-i},\underline{-k})\text{, and }(\underline{j},\underline{k},\underline{-j},\underline{-k}).
    \]
Any pair of these \(4\)-cycles intersect at two vertices and each \(4\)-cycle can be viewed as the equator of an octohedron in \(\simpcmplx(G)\) whose poles are \(\overline{\pm u}\) for some \(u\in\set{i,j,k}\). For example, the first of the aforementioned \(4\)-cycles bounds an octohedron whose poles are \(\overline{\pm k}\).

Note that \(\simpcmplx(G)\) is the simplicial complex Herman and Pakianathan called \(X(Q_8)\) \cite[p.18]{herman}. The geometric realization of \(\simpcmplx(G)\) consists of three \(2\)-spheres, each pair of which is glued at two points. Since \(\oser(G)\) is the geometric realization of \(\simpcmplx(G)\) minus its \(0\)-skeleton, we find that \(\oser(G)\) is the disjoint union of three copies of a \(2\)-sphere which has had \(6\) points removed. This is not the desingularized complex \(Y(Q_8)\) of Herman and Pakianathan, which is simply the disjoint union of three \(2\)-spheres.
\end{example}

Our next example uses the alternating \(3\)-quasigroup from \autoref{ex:order_five}.

\begin{example}
\label{ex:order_five_open}
Let \(A\) be the alternating \(3\)-quasigroup of order \(5\) introduced in \autoref{ex:order_five}. We have that
    \[
        \nct(A)=\set[(a,b,c)]{\set{a,b,c}\in\binom{A}{3}}
    \]
so
    \[
        \inp(A)=A=\out(A).
    \]
We see that
    \[
        \vertices(\simpcmplx(A))=\set[\underline{u}]{u\in A}\cup\set[\overline{u}]{u\in A}.
    \]
For chosen representatives of the orbits of the alternating group on \(\nct(A)\) \autoref{fig:order_five_facets} lists the corresponding facets of \(\simpcmplx(A)\).

\begin{table}
    \centering
    \begin{tabular}{c|c}
            \((a,b,c)\in\nct(A)\) & \(\sigma\in\simpcmplx(A)\) \\ \hline
            \((0,1,2)\) & \(\set{\underline{0},\underline{1},\underline{2},\overline{3}}\) \\
            \((1,2,3)\) & \(\set{\underline{1},\underline{2},\underline{3},\overline{4}}\) \\
            \((2,3,4)\) & \(\set{\underline{2},\underline{3},\underline{4},\overline{0}}\) \\
            \((3,4,0)\) & \(\set{\underline{3},\underline{4},\underline{0},\overline{1}}\) \\
            \((4,0,1)\) & \(\set{\underline{4},\underline{0},\underline{1},\overline{2}}\) \\
            \((0,2,1)\) & \(\set{\underline{0},\underline{2},\underline{1},\overline{4}}\) \\
            \((1,3,2)\) & \(\set{\underline{1},\underline{3},\underline{2},\overline{0}}\) \\
            \((2,4,3)\) & \(\set{\underline{2},\underline{4},\underline{3},\overline{1}}\) \\
            \((3,0,4)\) & \(\set{\underline{3},\underline{0},\underline{4},\overline{2}}\) \\
            \((4,1,0)\) & \(\set{\underline{4},\underline{1},\underline{0},\overline{3}}\) \\
            \((0,1,3)\) & \(\set{\underline{0},\underline{1},\underline{3},\overline{4}}\) \\
            \((1,2,4)\) & \(\set{\underline{1},\underline{2},\underline{4},\overline{0}}\) \\
            \((2,3,0)\) & \(\set{\underline{2},\underline{3},\underline{0},\overline{1}}\) \\
            \((3,4,1)\) & \(\set{\underline{3},\underline{4},\underline{1},\overline{2}}\) \\
            \((4,0,2)\) & \(\set{\underline{4},\underline{0},\underline{2},\overline{3}}\) \\
            \((0,3,1)\) & \(\set{\underline{0},\underline{3},\underline{1},\overline{2}}\) \\
            \((1,4,2)\) & \(\set{\underline{1},\underline{4},\underline{2},\overline{3}}\) \\
            \((2,0,3)\) & \(\set{\underline{2},\underline{0},\underline{3},\overline{4}}\) \\
            \((3,1,4)\) & \(\set{\underline{3},\underline{1},\underline{4},\overline{0}}\) \\
            \((4,2,0)\) & \(\set{\underline{4},\underline{2},\underline{0},\overline{1}}\)
    \end{tabular}
    
    \medskip
    \medskip
    \caption{Facets of \(\simpcmplx(A)\)}
    \label{fig:order_five_facets}
\end{table}

Observe that if
    \[
        \set{u_0,u_1,u_2,u_3,u_4}=\set{0,1,2,3,4}
    \]
then \(\simpcmplx(A)\) has a subcomplex \(S\) whose facets form the set
    \[
        \set[\set{\underline{u}_i,\underline{u}_j,\underline{u}_k,\overline{u}_4}]{\set{i,j,k}\in\binom{\set{0,1,2,3}}{3}}.
    \]
Note that \(\geo(S)\) is homeomorphic to the \(3\)-simplex whose vertices are \(\underline{u}_0\), \(\underline{u}_1\), \(\underline{u}_2\), and \(\underline{u}_3\). It follows that \(\simpcmplx(A)\) is the subdivided complex obtained by adding a single vertex to the center of each facet of the boundary of the \(4\)-simplex whose vertices are the \(\underline{u}_i\) for \(i\in\set{0,1,2,3,4}\). This implies that \(\geo(\simpcmplx(A))\) is homeomorphic to the \(3\)-sphere. Thus, \(\oser(A)\) is homeomorphic to the \(3\)-sphere minus the \(1\)-skeleton of \(\simpcmplx(A)\). That \(1\)-skeleton is the graph pictured in \autoref{fig:order_five_skeleton}, which is homotopy equivalent to the join of \(21\) circles.

    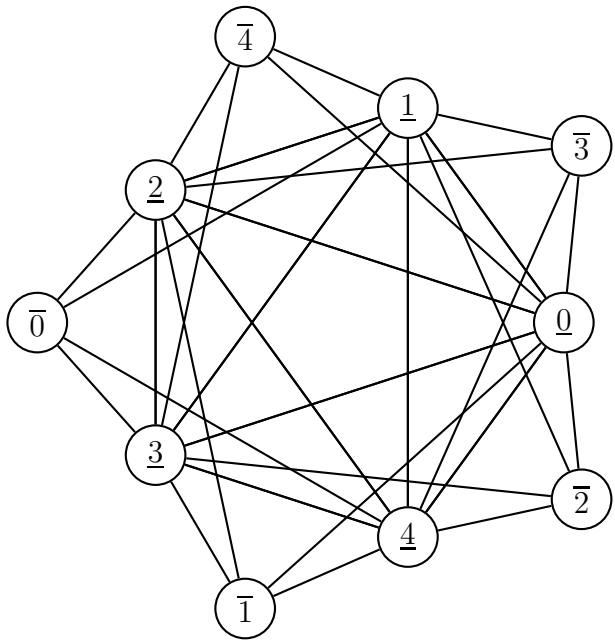
\begin{figure}
        \centering
            \begin{tikzpicture}[thick]
                \tikzstyle{every node}=[draw,circle]
                \node (a0) at (0:3) {\(\underline{0}\)};
                \node (a1) at (360/5:3) {\(\underline{1}\)};
                \node (a2) at (2*360/5:3) {\(\underline{2}\)};
                \node (a3) at (3*360/5:3) {\(\underline{3}\)};
                \node (a4) at (4*360/5:3) {\(\underline{4}\)};
                \node (b0) at (0+180:4) {\(\overline{0}\)};
                \node (b1) at (360/5+180:4) {\(\overline{1}\)};
                \node (b2) at (2*360/5+180:4) {\(\overline{2}\)};
                \node (b3) at (3*360/5+180:4) {\(\overline{3}\)};
                \node (b4) at (4*360/5+180:4) {\(\overline{4}\)};
                \foreach \x in {1,2,3,4}
                {
                    \draw (a0) -- (a\x);
                    \draw (b0) -- (a\x);
                };
                \foreach \x in {0,2,3,4}
                {
                    \draw (a1) -- (a\x);
                    \draw (b1) -- (a\x);
                };
                \foreach \x in {0,1,3,4}
                {
                    \draw (a2) -- (a\x);
                    \draw (b2) -- (a\x);
                };
                \foreach \x in {0,1,2,4}
                {
                    \draw (a3) -- (a\x);
                    \draw (b3) -- (a\x);
                };
                \foreach \x in {0,1,2,3}
                {
                    \draw (a4) -- (a\x);
                    \draw (b4) -- (a\x);
                };
            \end{tikzpicture}
        \caption{The \(1\)-skeleton of \(\simpcmplx(A)\)}
        \label{fig:order_five_skeleton}
    \end{figure}
\end{example}

For our penultimate example we consider the order \(6\) alternating \(3\)-quasigroup from \autoref{ex:order_six}.

\begin{example}
\label{ex:order_six_open}
Let \(A\) be the ternary alternating quasigroup \(U\boxtimes_\alpha V\) introduced in \autoref{ex:order_six}. We have that
    \[
        \nct(A)=\set[((u_1,v_1),(u_2,v_2),(u_3,v_3))\in(U\times V)^3]{\set{u_1,u_2,u_3}=\set{0,1,2}}.
    \]
so
    \[
        \inp(A)=A
    \]
and
    \[
        \out(A)=\set{(0,0),(0,1)}.
    \]
We see that
    \[
        \vertices(\simpcmplx(A))=\set[\underline{uv}]{u\in\Z/3\Z\text{ and }v\in\Z/2\Z}\cup\set{\overline{00},\overline{01}}.
    \]
For chosen representative of the orbits of the alternating group on \(\nct(A)\) \autoref{fig:order_six_facets} lists the corresponding facets of \(\simpcmplx(A)\).

    \begin{table}
        \centering
        \begin{tabular}{c|c}
                \(((u_1,v_1),(u_2,v_2),(u_3,v_3))\in\nct(A)\) & \(\sigma\in\simpcmplx(A)\) \\ \hline
                \(((0,0),(1,0),(2,0))\) & \(\set{\underline{00},\underline{10},\underline{20},\overline{00}}\) \\
                \(((0,0),(1,0),(2,1))\) &
                \(\set{\underline{00},\underline{10},\underline{21},\overline{01}}\) \\
                \(((0,0),(1,1),(2,0))\) &
                \(\set{\underline{00},\underline{11},\underline{20},\overline{01}}\) \\
                \(((0,0),(1,1),(2,1))\) &
                \(\set{\underline{00},\underline{11},\underline{21},\overline{00}}\) \\
                \(((0,1),(1,0),(2,0))\) &
                \(\set{\underline{01},\underline{10},\underline{20},\overline{01}}\) \\
                \(((0,1),(1,0),(2,1))\) &
                \(\set{\underline{01},\underline{10},\underline{21},\overline{00}}\) \\
                \(((0,1),(1,1),(2,0))\) &
                \(\set{\underline{01},\underline{11},\underline{20},\overline{00}}\) \\
                \(((0,1),(1,1),(2,1))\) &
                \(\set{\underline{01},\underline{11},\underline{21},\overline{01}}\) \\
                \(((0,0),(2,0),(1,0))\) &
                \(\set{\underline{00},\underline{20},\underline{10},\overline{01}}\) \\
                \(((0,0),(2,0),(1,1))\) &
                \(\set{\underline{00},\underline{20},\underline{11},\overline{00}}\) \\
                \(((0,0),(2,1),(1,0))\) &
                \(\set{\underline{00},\underline{21},\underline{10},\overline{00}}\) \\
                \(((0,0),(2,1),(1,1))\) &
                \(\set{\underline{00},\underline{21},\underline{11},\overline{01}}\) \\
                \(((0,1),(2,0),(1,0))\) &
                \(\set{\underline{01},\underline{20},\underline{10},\overline{00}}\) \\
                \(((0,1),(2,0),(1,1))\) &
                \(\set{\underline{01},\underline{20},\underline{11},\overline{01}}\) \\
                \(((0,1),(2,1),(1,0))\) &
                \(\set{\underline{01},\underline{21},\underline{10},\overline{01}}\) \\
                \(((0,1),(2,1),(1,1))\) &
                \(\set{\underline{01},\underline{21},\underline{11},\overline{00}}\)
        \end{tabular}
        
        \medskip
        \medskip
        \caption{Facets of \(\simpcmplx(A)\)}
        \label{fig:order_six_facets}
    \end{table}
    
The induced subcomplex on \(\set[\underline{uv}]{u\in\Z/3\Z\text{ and }v\in\Z/2\Z}\) is an octohedron, as one may see by taking as the equator the vertices \(\underline{00}\), \(\underline{20}\), \(\underline{01}\), and \(\underline{21}\), in that order, and taking the poles to be \(\overline{10}\) and \(\overline{11}\). It follows that \(\geo(\simpcmplx(A))\) is a \(3\)-sphere whose hemispheres are the two cones over this octohedron with cone points \(\overline{00}\) and \(\overline{01}\). Thus, \(\oser(A)\) is homeomorphic to the \(3\)-sphere minus the \(1\)-skeleton of \(\simpcmplx(A)\). That \(1\)-skeleton is the graph pictured in \autoref{fig:order_six_skeleton}, which is homotopy equivalent to the join of \(17\) circles.

    \begin{figure}
        \centering
            \begin{tikzpicture}[thick]
                \tikzstyle{every node}=[draw,circle]
                \node (a00) at (360/8:4) {\(\underline{00}\)};
                \node (a20) at (3*360/8:4) {\(\underline{20}\)};
                \node (a01) at (5*360/8:4) {\(\underline{01}\)};
                \node (a21) at (7*360/8:4) {\(\underline{21}\)};
                \node (a10) at (0,1) {\(\underline{10}\)};
                \node (a11) at (0,-1) {\(\underline{11}\)};
                \node (b00) at (5,0) {\(\overline{00}\)};
                \node (b01) at (-5,0) {\(\overline{01}\)};
                \draw (a00) -- (a20) -- (a01) -- (a21) -- (a00);
                \draw (a00) -- (a10) -- (a20);
                \draw (a01) -- (a10) -- (a21);
                \draw (a00) -- (a11) -- (a20);
                \draw (a01) -- (a11) -- (a21);
                \foreach \x in {0,1,2}
                {
                    \foreach \y in {0,1}
                    {
                        \draw (a\x\y) -- (b00);
                        \draw (a\x\y) -- (b01);
                    };
                };
            \end{tikzpicture}
        \caption{The \(1\)-skeleton of \(\simpcmplx(A)\)}
        \label{fig:order_six_skeleton}
    \end{figure}
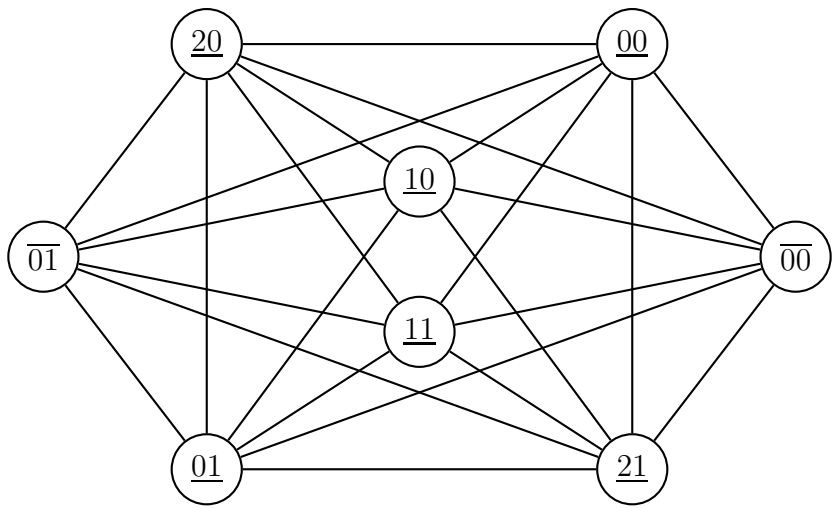
\end{example}

Our last example is the most degenerate, but it's worth noting that this corner case is still defined.

\begin{example}
Suppose that \(A\) is a commutative \(n\)-quasigroup. We have that \(\oser(A)\) is the empty manifold. To see this, note that \(\nct(A)=\varnothing\), which implies that \(\inp(A)=\varnothing\) and \(\out(A)=\varnothing\). It follows that \(\simpcmplx(A)=\varnothing\) and hence
    \[
        (\ogeo\circ\simpcmplx)(A)
    \]
is the empty topological space with no points.
\end{example}

\subsection{The graph retract}
The manifolds \(\oser(A)\) are relatively unstructured up to homotopy.

\begin{defn}[NC graph]
Given an alternating \(n\)-quasigroup \(A\) the \emph{NC graph} of \(A\) is the simple graph
    \[
        \ncgr(A)=(\ncvert(A),\ncedge(A))
    \]
where
    \[
        \ncvert(A)=\orb_{\alt_n}(\nct(A))
    \]
and we define \(\ncedge(A)\) to consist of all pairs
    \[
        \set{a/\alt_n,b/\alt_n}\in\binom{\ncvert(A)}{2}
    \]
such that either
    \begin{enumerate}
        \item we have \(\set{a_1,\dots,a_n}=\set{b_1,\dots,b_n}\) or
        \item \(f(a)=f(b)\) and
                \[
                    \abs{\set{a_1,\dots,a_n}\cap\set{b_1,\dots,b_n}}=n-1.
                \]
    \end{enumerate}
\end{defn}

\begin{prop}
For any alternating \(n\)-quasigroup \(A\) we have an embedding
    \[
        \iota\colon\geo(\ncgr(A))\hookrightarrow\oser(A)
    \]
such that \(\im(\iota)\) is a strong deformation retract of \(\oser(A)\).
\end{prop}

\begin{proof}
Given
    \[
        \set{a/\alt_n,b/\alt_n}\in\ncedge(A)
    \]
and \(\gamma\in\left[\frac{1}{2},1\right]\) we define
    \[
        \iota(\gamma a/\alt_n+(1-\gamma) b/\alt_n)=\frac{\gamma}{n+1}\left(\sum_{i=1}^n\underline{a}_i+\overline{f(a)}\right)+\frac{1-\gamma}{n}\sum_{i=1}^nc_i
    \]
where
    \[
        \set{c_1,\dots,c_n}=\set{\underline{a}_1,\dots,\underline{a}_n,\overline{f(a)}}\cap\set{\underline{b}_1,\dots,\underline{b}_n,\overline{f(b)}}.
    \]

Edges from \(\ncedge(A)\) are mapped by \(\iota\) to piecewise linear curves between the midpoints of the geometric realization of facets in \(\simpcmplx(A)\) which intersect at an \((n-1)\)-face.

Define a homotopy
    \[
        h\colon\oser(A)\times[0,1]\to\oser(A)
    \]
as follows. Let \(\set{b_1,\dots,b_{n+1}}\) be the vertices of a facet of \(\simpcmplx(A)\) and consider the maximal flag
    \[
        \set{b_1}\subseteq\set{b_1,b_2}\subseteq\cdots\subseteq\set{b_1,\dots,b_{n+1}}
    \]
associated with the given labeling on the \(b_i\). Define
    \[
        c_k=\frac{1}{k}\sum_{i=1}^kb_k
    \]
and note that the \(c_k\) are the vertices of the facet of the barycentric subdivision of the simplex with vertices \(\set{b_1,\dots,b_{n+1}}\) corresponding to the flag in question. When
    \[
        x\in\cvx(\set{c_1,\dots,c_{n+1}})
    \]
let \(P\) be the affine span of
    \[
        \set{c_1,\dots,c_{n-1},x}
    \]
and let \(y\) be the unique point of intersection between \(P\) and \(\cvx(\set{c_n,c_{n+1}})\). We set
    \[
        h(x,t)=(1-t)x+ty.
    \]
Note that if \(x\in\im(\iota)\) then \(y=x\) and \(h(x,t)=x\) for all time \(t\).
\end{proof}

\subsection{Examples of the graph retract}
\label{subsec:graph_retract_examples}
We give several examples of \(\ncgr(A)\) for various choices of \(A\). We can be a bit more expansive than in the analogous \autoref{subsec:examples_open_serenation} as there are fewer topological considerations here.

Our first example is that of the quaternion group of \autoref{ex:quaternion_open}.

\begin{example}
\label{ex:quaternion_graph}
Let \(G\) be the quaternion group of order \(8\). Note that we have one vertex of \(\ncgr(G)\) for each facet of \(\simpcmplx(G)\) so we have that
    \[
        \ncvert(G)=\set[(\pm u,\pm v)/\alt_2]{\set{u,v}\in\binom{\set{i,j,k}}{2}}.
    \]
Choosing orbit representatives and suppressing the obvious isomorphism we take
    \[
        \ncvert(G)=\set[(\pm u,\pm v)]{\set{u,v}\in\binom{\set{i,j,k}}{2}}.
    \]
We see that \(\ncgr(G)\) is \(3\)-regular with the neighbors of \((x,y)\) being
    \[
        (y,x)\text{, }(xyx^{-1},x)\text{, and }(y,y^{-1}xy).
    \]
The resulting graph is pictured in \autoref{fig:quaternion_graph}. We see that
    \[
        \ncgr(G)\cong Q_3\sqcup Q_3\sqcup Q_3
    \]
where \(Q_3\) is the \(3\)-cube graph.

    \begin{figure}
        \centering
            \resizebox{\textwidth}{!}{
            \begin{tikzpicture}[thick]
                \tikzstyle{every node}=[draw,circle]
                \node (i0j0) at (0,0) {\((i,j)\)};
                \node (i0j1) at (2,4) {\((i,-j)\)};
                \node (i1j0) at (4,2) {\((-i,j)\)};
                \node (i1j1) at (6,6) {\((-i,-j)\)};
                \node (j0i0) at (2,2) {\((j,i)\)};
                \node (j0i1) at (6,0) {\((j,-i)\)};
                \node (j1i0) at (0,6) {\((-j,i)\)};
                \node (j1i1) at (4,4) {\((-j,-i)\)};
                \draw (i0j0) -- (j0i1) -- (i1j1) -- (j1i0) -- (i0j0);
                \draw (j0i0) -- (i1j0) -- (j1i1) -- (i0j1) -- (j0i0);
                \draw (i0j0) -- (j0i0);
                \draw (j0i1) -- (i1j0);
                \draw (i1j1) -- (j1i1);
                \draw (j1i0) -- (i0j1);
                \node (i0k0) at (8,0) {\((i,k)\)};
                \node (i0k1) at (10,4) {\((i,-k)\)};
                \node (i1k0) at (12,2) {\((-i,k)\)};
                \node (i1k1) at (14,6) {\((-i,-k)\)};
                \node (k0i0) at (10,2) {\((k,i)\)};
                \node (k0i1) at (14,0) {\((k,-i)\)};
                \node (k1i0) at (8,6) {\((-k,i)\)};
                \node (k1i1) at (12,4) {\((-k,-i)\)};
                \draw (i0k0) -- (k0i1) -- (i1k1) -- (k1i0) -- (i0k0);
                \draw (k0i0) -- (i1k0) -- (k1i1) -- (i0k1) -- (k0i0);
                \draw (i0k0) -- (k0i0);
                \draw (k0i1) -- (i1k0);
                \draw (i1k1) -- (k1i1);
                \draw (k1i0) -- (i0k1);
                \node (j0k0) at (4,-8) {\((j,k)\)};
                \node (j0k1) at (6,-4) {\((j,-k)\)};
                \node (j1k0) at (8,-6) {\((-j,k)\)};
                \node (j1k1) at (10,-2) {\((-j,-k)\)};
                \node (k0j0) at (6,-6) {\((k,j)\)};
                \node (k0j1) at (10,-8) {\((k,-j)\)};
                \node (k1j0) at (4,-2) {\((-k,j)\)};
                \node (k1j1) at (8,-4) {\((-k,-j)\)};
                \draw (j0k0) -- (k0j1) -- (j1k1) -- (k1j0) -- (j0k0);
                \draw (k0j0) -- (j1k0) -- (k1j1) -- (j0k1) -- (k0j0);
                \draw (j0k0) -- (k0j0);
                \draw (k0j1) -- (j1k0);
                \draw (j1k1) -- (k1j1);
                \draw (k1j0) -- (j0k1);
            \end{tikzpicture}
            }
        \caption{The graph \(\ncgr(G)\)}
        \label{fig:quaternion_graph}
    \end{figure}
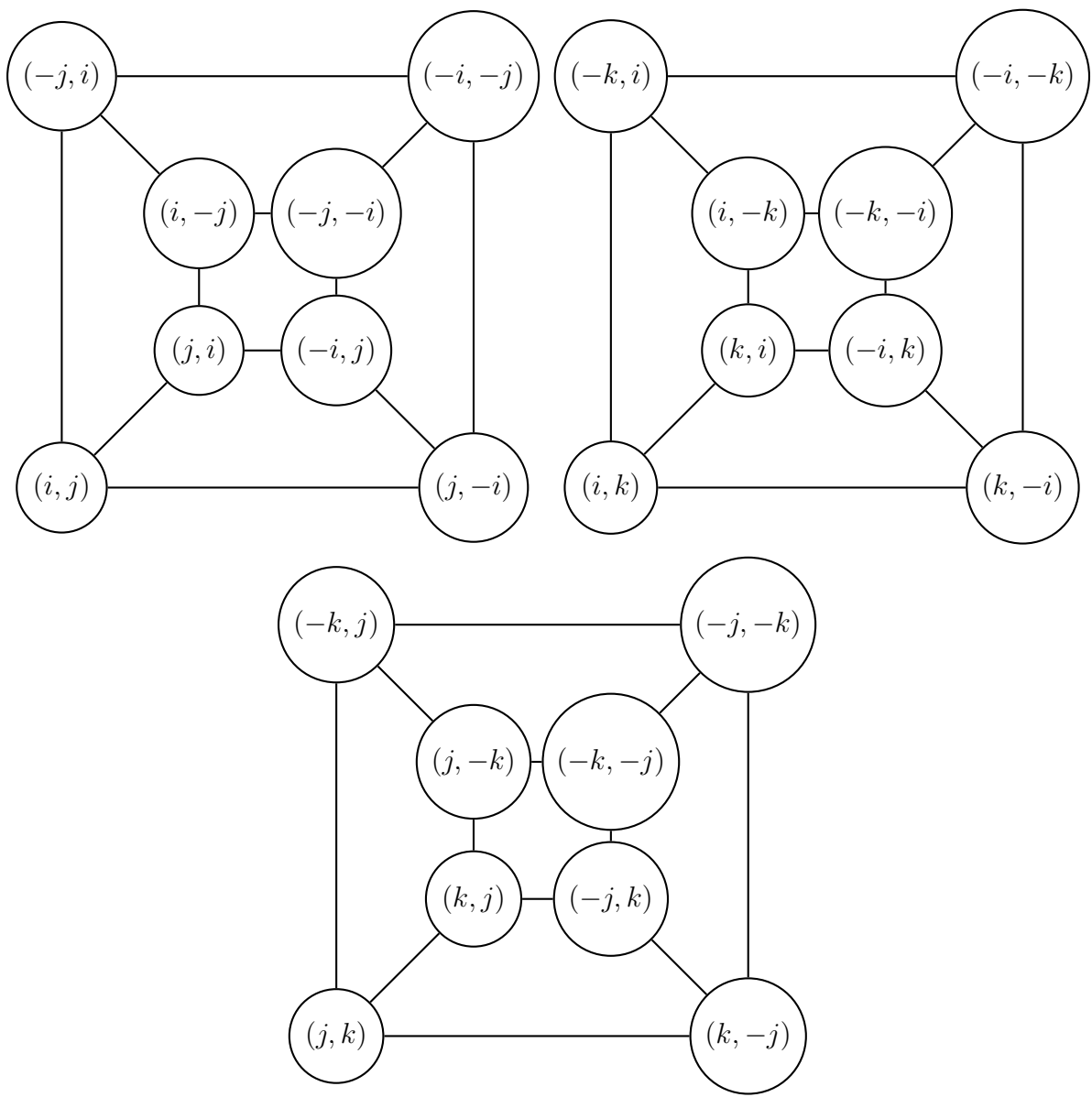
\end{example}

Our next example concerns the order \(5\) quasigroup from \autoref{ex:order_five} and \autoref{ex:order_five_open}.

\begin{example}
\label{ex:order_five_graph}
Let \(A\) be the alternating \(3\)-quasigroup of order \(5\) from the examples indicated above. As in our previous example of an NC graph we choose orbit representatives under the action of \(\alt_3\) and suppress the obvious isomorphism to say that
    \begin{align*}
        \ncvert(A)=\{ & 012,123,234,340,401,021,132,243,304,410, \\
        & 013,124,230,341,402,031,142,203,314,420\}.
    \end{align*}
Since \(A\) is ternary we have that \(\ncgr(A)\) is \(4\)-regular. We could give a sort of conjugacy formula for the neighbors of a vertex analogous to that from \autoref{ex:quaternion_graph} but in this case it's easier to just directly examine \autoref{fig:order_five_facets} in order to see that \(\ncgr(A)\) is the graph pictured in \autoref{fig:order_five_graph}.

    \begin{figure}
        \centering
            \resizebox{\textwidth}{!}{
            \begin{tikzpicture}[thick]
                \tikzstyle{every node}=[draw,circle]
                \node (012) at (-1,1) {\(012\)};
                \node (410) at (1,1) {\(410\)};
                \node (402) at (1,-1) {\(402\)};
                \node (142) at (-1,-1) {\(142\)};
                \draw (012) -- (410) -- (402) -- (142) -- (012);
                \draw (012) -- (402);
                \draw (410) -- (142);
                \node (021) at (-2,2) {\(021\)};
                \node (013) at (-2,4) {\(013\)};
                \node (203) at (-4,4) {\(203\)};
                \node (123) at (-4,2) {\(123\)};
                \draw (021) -- (013) -- (203) -- (123) -- (021);
                \draw (021) -- (203);
                \draw (013) -- (123);
                \node (401) at (2,2) {\(401\)};
                \node (304) at (4,2) {\(304\)};
                \node (341) at (4,4) {\(341\)};
                \node (031) at (2,4) {\(031\)};
                \draw (401) -- (304) -- (341) -- (031) -- (401);
                \draw (401) -- (341);
                \draw (304) -- (031);
                \node (124) at (-2,-2) {\(124\)};
                \node (132) at (-4,-2) {\(132\)};
                \node (314) at (-4,-4) {\(314\)};
                \node (234) at (-2,-4) {\(234\)};
                \draw (124) -- (132) -- (314) -- (234) -- (124);
                \draw (124) -- (314);
                \draw (132) -- (234);
                \node (420) at (2,-2) {\(420\)};
                \node (340) at (4,-2) {\(340\)};
                \node (230) at (4,-4) {\(230\)};
                \node (243) at (2,-4) {\(243\)};
                \draw (420) -- (340) -- (230) -- (243) -- (420);
                \draw (420) -- (230);
                \draw (340) -- (243);
                \draw (012) -- (021);
                \draw (142) -- (124);
                \draw (402) -- (420);
                \draw (410) -- (401);
                \draw (132) -- (123);
                \draw (013) -- (031);
                \draw (304) -- (340);
                \draw (243) -- (234);
                \draw (314) to [out=135,in=270] (-5,-3);
                \draw (-5,-3) -- (-5,4);
                \draw (-5,4) to [out=90,in=180] (-4,5);
                \draw (-4,5) to [out=0,in=160] (341);
                \draw (230) to [out=45,in=270] (5,-3);
                \draw (5,-3) -- (5,4);
                \draw (5,4) to [out=90,in=0] (4,5);
                \draw (4,5) to [out=180,in=20] (203);
            \end{tikzpicture}
            }
        \caption{The graph \(\ncgr(A)\)}
        \label{fig:order_five_graph}
    \end{figure}
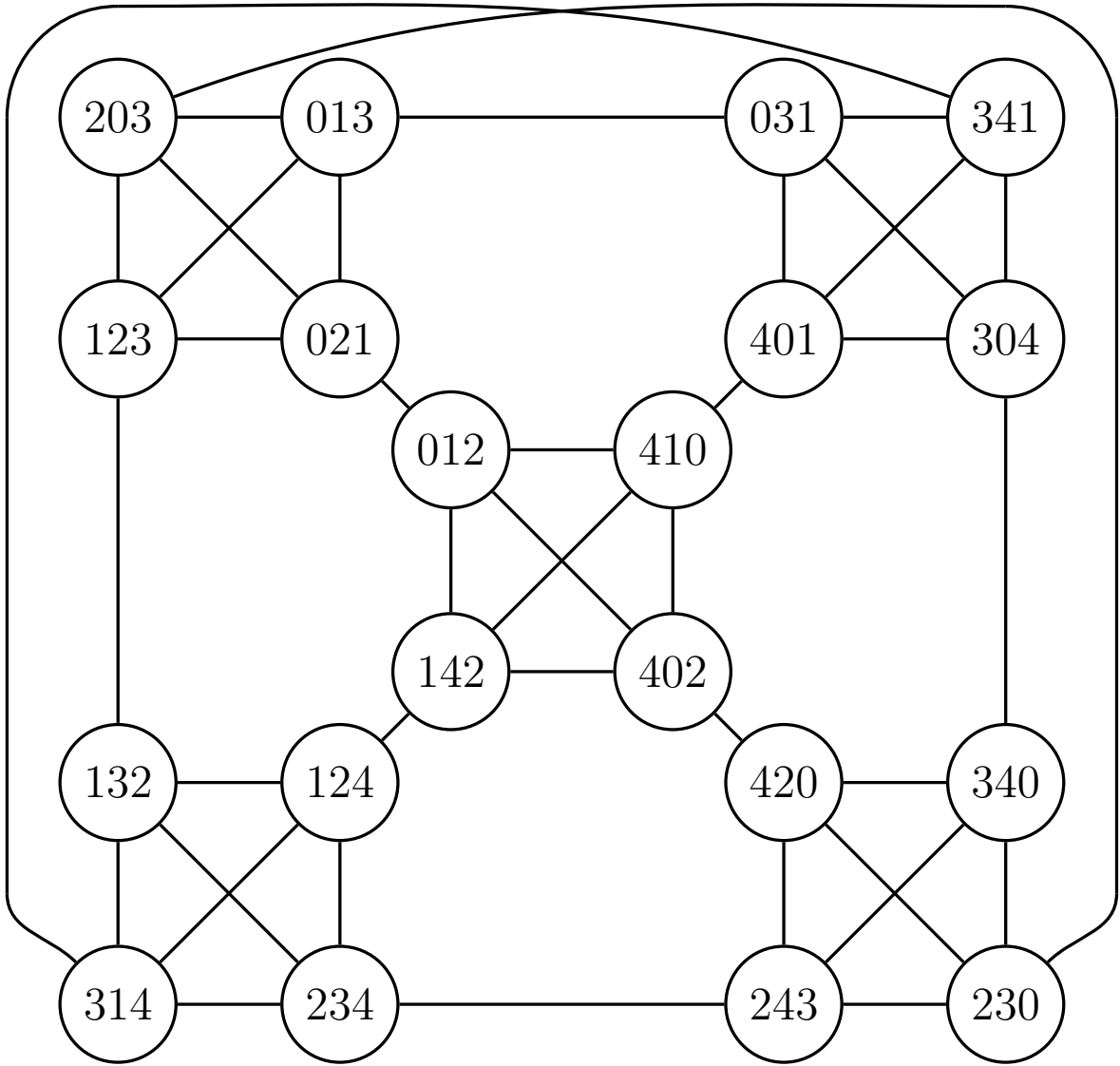
\end{example}

The case of the order \(6\) quasigroup from \autoref{ex:order_six} and \autoref{ex:order_six_open} is similar.

\begin{example}
\label{ex:order_six_graph}
Let \(A\) be the alternating \(3\)-quasigroup of order \(6\) from the examples indicated above. Again we choose orbit representatives and declare that
    \begin{align*}
        \ncvert(A)=\{ & (00,10,20),(00,10,21),(00,11,20),(00,11,21), \\
        & (01,10,20),(01,10,21),(01,11,20),(01,11,21), \\
        & (00,20,10),(00,20,11),(00,21,10),(00,21,11), \\
        & (01,20,10),(01,20,11),(01,21,10),(01,21,11)\}.
    \end{align*}

In this case we see that \(\ncgr(A)\), which is pictured in \autoref{fig:order_six_graph}, is isomorphic to \(C_4\square C_4\) where \(C_4\) is the \(4\)-cycle graph and \(\square\) is the Cartesian product of graphs. Note also that \(C_4\square C_4\cong Q_4\), where \(Q_4\) is the \(4\)-cube graph.

    \begin{figure}
        \centering
            \resizebox{\textwidth}{!}{
            \begin{tikzpicture}[thick]
                \tikzstyle{every node}=[draw,circle]
                \node (011120) at (2,2) {\((01,11,20)\)};
                \node (012010) at (-2,2) {\((01,20,10)\)};
                \node (011020) at (-2,-2) {\((01,10,20)\)};
                \node (012011) at (2,-2) {\((01,20,11)\)};
                \draw (011120) -- (012010) -- (011020) -- (012011) -- (011120);
                \node (002011) at (4,4) {\((00,20,11)\)};
                \node (001020) at (-4,4) {\((00,10,20)\)};
                \node (002010) at (-4,-4) {\((00,20,10)\)};
                \node (001120) at (4,-4) {\((00,11,20)\)};
                \draw (002011) -- (001020) -- (002010) -- (001120) -- (002011);
                \node (001121) at (6,6) {\((00,11,21)\)};
                \node (002110) at (-6,6) {\((00,21,10)\)};
                \node (001021) at (-6,-6) {\((00,10,21)\)};
                \node (002111) at (6,-6) {\((00,21,11)\)};
                \draw (001121) -- (002110) -- (001021) -- (002111) -- (001121);
                \node (012111) at (8,8) {\((01,21,11)\)};
                \node (011021) at (-8,8) {\((01,10,21)\)};
                \node (012110) at (-8,-8) {\((01,21,10)\)};
                \node (011121) at (8,-8) {\((01,11,21)\)};
                \draw (012111) -- (011021) -- (012110) -- (011121) -- (012111);
                \draw (012010) -- (001020) -- (002110) -- (011021);
                \draw (011120) -- (002011) -- (001121) -- (012111);
                \draw (011020) -- (002010) -- (001021) -- (012110);
                \draw (012011) -- (001120) -- (002111) -- (011121);
                \draw (011120) to [out=90,in=180] (012111);
                \draw (012010) to [out=180,in=270] (011021);
                \draw (011020) to [out=270,in=0] (012110);
                \draw (012011) to [out=0,in=90] (011121);
            \end{tikzpicture}
            }
        \caption{The graph \(\ncgr(A)\)}
        \label{fig:order_six_graph}
    \end{figure}
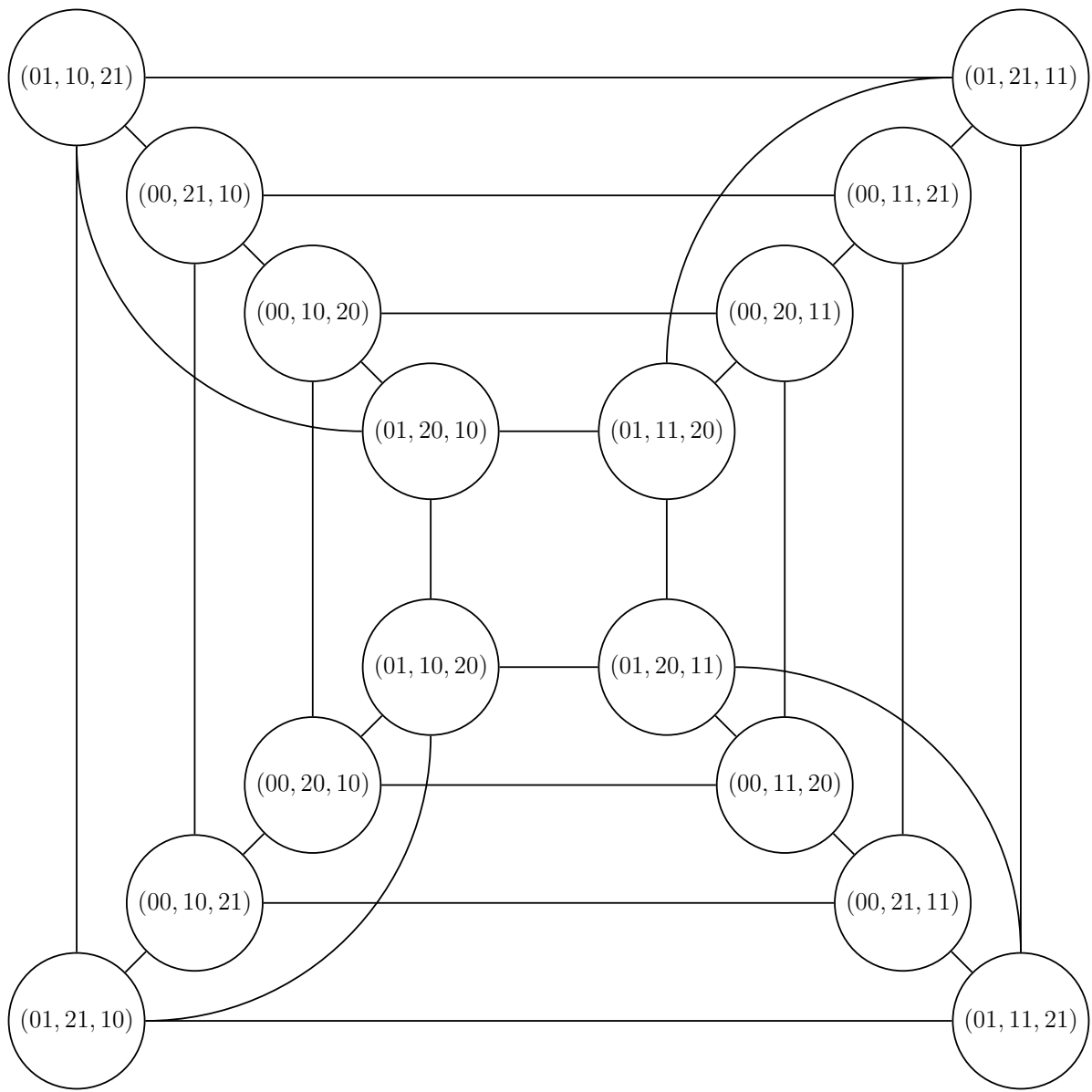
\end{example}

Our last example concerns the quasigroups \(\F^{(n)}\) of \autoref{ex:field}.

\begin{example}
Fix an odd prime power \(q\) and observe that
    \[
        ((u_1,v_1),\dots,(u_n,v_n))\in\nct(\F_q^{(n)})
    \]
if and only if
    \[
        \det(u_1,\dots,u_n)\neq0.
    \]
It follows that
    \[
        \nct(\F_q^{(n)})\cong\operatorname{GL}_n(\F_q)\times\F_q^n
    \]
where \(\operatorname{GL}_n(\F_q)\) is the set of invertible \(n\times n\) matrices with entries in \(\F_q\). We find that \(\ncgr(\F_q^{(n)})\) is an \((n+1)\)-regular graph with
    \begin{align*}
        \abs{\orb_{\alt_n}(\nct(\F_q^{(n)}))} &= \frac{2}{n!}\abs{\nct(\F_q^{(n)})} \\
        &= \frac{2}{n!}\abs{\operatorname{GL}_n(\F_q)}\abs{\F_q^n} \\
        &= \frac{2q^n}{n!}\prod_{k=1}^{n-1}(q^n-q^k)
    \end{align*}
vertices.
\end{example}

\subsection{NC graphs and Johnson graphs}
Finite NC graphs are quite structured, as they are regular and are induced subgraphs of Johnson graphs. For an introduction to Johnson graphs, see \cite{godsil}. Similar comments hold in the infinite case.

\begin{prop}
Let \(A\) be a finite alternating \(n\)-quasigroup and let
    \[
        s=\abs{\vertices(\simpcmplx(A))}=\abs{\inp(A)}+\abs{\out(A)}.
    \]
We have that \(\ncgr(A)\) is an induced subgraph of the Johnson graph \(J(s,n+1,n)\).
\end{prop}

\begin{proof}
Let \(\psi\colon\vertices(\simpcmplx(A))\to[s]\) be a bijection and define a graph homomorphism
    \[
        h\colon\ncgr(A)\to J(s,n+1,n)
    \]
by
    \[
        h(a/\alt_n)=\set{\psi(\underline{a}_1),\dots,\psi(\underline{a}_n),\psi(\overline{f(a)})}.
    \]
Observe that \(h\) is an embedding and that \(\set{a/\alt_n,b/\alt_n}\in\ncedge(A)\) if and only if \(h(a/\alt_n)\) and \(h(b/\alt_n)\) have exactly \(n\) elements in common.
\end{proof}

Since \(\ncgr(A)\) is an induced subgraph of \(J(s,n+1,n)\) when \(A\) is a finite alternating \(n\)-quasigroup we can use interlacing \cite[Theorem 9.1.1]{godsil} to understand the spectrum of \(\ncgr(A)\) given the spectrum of \(J(s,n+1,n)\). The graph \(\ncgr(A)\) has \(\abs{\simpcmplx(A)}\) vertices, so interlacing will give stronger results the larger the ratio
    \[
        \frac{\abs{\simpcmplx(A)}}{\abs{\binom{\vertices(\simpcmplx(A))}{n+1}}}
    \]
is.

\subsection{The Riemannian metric}
For any alternating \(n\)-quasigroup \(A\) we have that \(\oser(A)\) carries a canonical metric. Let \(\delta\) denote the Kronecker delta function.

\begin{defn}[Standard metric]
Let \(A\) be an alternating \(n\)-quasigroup. The \emph{standard metric} \(g\) on \(\oser(A)\) is given by bilinear extension of the following rules:
    \begin{enumerate}
        \item When \(x\in\ocvx\left(\set{\underline{a}_1,\dots,\underline{a}_n,\overline{f(a)}}\right)\) we set
            \[
                g_x(\underline{a}_i-\overline{f(a)},\underline{a}_j-\overline{f(a)})=1+\delta_{ij}.
            \]
        \item When \(x\in\ocvx\left(\set{\underline{a}_1,\dots,\underline{a}_{n-1},\overline{f(a)}}\right)\) we set
            \[
                g_x(\underline{a}_i,\underline{a}_j)=1+\delta_{ij},
            \]
            \[
                g_x(\underline{a}_i,\overline{f(a)})=1,
            \]
            and
            \[
                g_x(\overline{f(a)},\overline{f(a)})=2.
            \]
        \item When \(x\in\ocvx\left(\set{\underline{a}_1,\dots,\underline{a}_n}\right)\) we set
            \[
                g_x(\underline{a}_i,\underline{a}_j)=1+\delta_{ij}.
            \]
    \end{enumerate}
\end{defn}

This metric isn't too exciting in the sense that it always makes \(\oser(A)\) flat. Intuitively, the Riemannian manifold \((\oser(A),g)\) consists of glued copies of regular simplices whose edges all have length \(\sqrt{2}\).

\begin{prop}
    The Riemannian manifold \((\oser(A),g)\) is flat for any alternating \(n\)-quasigroup \(A\) when \(g\) is the standard metric.
\end{prop}

\begin{proof}
We show that \(g\) is constant on any given coordinate chart, which makes the metric tensor \(0\).

Given \(a=(a_1,\dots,a_n)\in\nct(A)\) consider the serene chart of input type \(\underline{\phi}_a\). Given \(k\in[n]\) define a tangent vector field \(v_k\colon\im(\underline{\phi}_a)\to T\oser(A)\) to \(\oser(A)\) on \(\im(\underline{\phi}_a)\) by setting
    \[
        (v_k)_x=\underline{a}_k-\overline{f(a)}
    \]
when \(x=\underline{\phi}_a(u)\) where \(\sum_{i=1}^nu_i<1\),
    \[
        (v_k)_x=\underline{a}_k
    \]
when \(x=\underline{\phi}_a(u)\) where \(\sum_{i=1}^nu_i=1\), and
    \[
        (v_k)_x=(\underline{a}_k-\overline{f(a')})-\frac{2}{n}\sum_{\ell=1}^n(\underline{a}_\ell-\overline{f(a')})
    \]
when \(x=\underline{\phi}_a(u)\) where \(\sum_{i=1}^nu_i>1\).

Note that at each point \(x\in\im(\underline{\phi}_a)\) we have that \(\set{(v_1)_x,\dots,(v_n)_x}\) is a basis for \(T_x\) and that the vector fields \(v_k\) pull back to the standard basis constant vector fields on \(\bipyr_n\).

The matrix of \(g\) with respect to this coordinate chart \(\underline{\phi}_a\) and basis is
    \[
        [g_x((v_i)_x,(v_j)_x)]=[1+\delta_{ij}]=J_n+I_n
    \]
where \(J_n\) is the \(n\times n\) matrix whose entries are all \(1\) and \(I_n\) is the \(n\times n\) identity matrix. Since the matrix of \(g\) is constant as a function of \(x\) we find that \(\oser(A)\) is flat when endowed with \(g\) in \(\im(\underline{\phi}_a)\).

    \begin{table}
        \centering
        \begin{tabular}{r|c|c}
            \((v_k)_x\) & \(k<n\) & \(k=n\) \\ \hline
            \multirow{2}{*}{\(\sum_{i=1}^nu_i<1\)} & \multirow{2}{*}{\((\underline{a}_k-\overline{f(a)})-(\underline{a}_n-\overline{f(a)})\)} & \multirow{2}{*}{\(-(\underline{a}_n-\overline{f(a)})\)} \\
            & & \\ \hline
            \multirow{2}{*}{\(\sum_{i=1}^nu_i=1\)} & \multirow{2}{*}{\(\underline{a}_k\)} & \multirow{2}{*}{\(\overline{f(a)}\)} \\
            & & \\ \hline
            \multirow{5}{*}{\(\sum_{i=1}^nu_i>1\)} & & \\
            & \((\underline{a}_i-\overline{f(a)})-\) & \(-\frac{2}{n}\sum_{\ell=1}^{n-1}(\underline{a}_\ell-\overline{f(a)})+\) \\
            & \(\frac{2}{n}\sum_{\ell=1}^{n-1}(\underline{a}_\ell-\overline{f(a)})+\) & \((\underline{a}_{n+1}-\overline{f(a)})\) \\
            & \((\underline{a}_{n+1}-\overline{f(a)})\) & \\
            & &
        \end{tabular}
        
        \medskip
        \medskip
        \caption{Output chart tangent vector fields}
        \label{fig:output_chart_tangent}
    \end{table}

Now consider the serene chart of output type \(\overline{\phi}_a\). Given \(k\in[n]\) define a tangent vector field \(v_k\colon\im(\overline{\phi}_a)\to T\oser(A)\) to \(\oser(A)\) on \(\im(\overline{\phi}_a)\) as in \autoref{fig:output_chart_tangent} where \(x=\overline{\phi}_a(u)\). By a slightly more involved calculation we again see that the matrix of \(g\) with respect to this coordinate chart \(\overline{\phi}_a\) and basis is
    \[
        [g_x((v_i)_x,(v_j)_x)]=[1+\delta_{ij}]=J_n+I_n,
    \]
so \(g\) is indeed constant on any coordinate chart.
\end{proof}

Recall that one can define the distance between two points in \(\oser(A)\) with respect to such a metric by defining the length \(L(\gamma)\) of a piecewise continuously differentiable curve \(\gamma\colon[0,1]\to\oser(A)\) to be
    \[
        L(\gamma)=\integ{0}{1}{\sqrt{g_{\gamma(t)}(\gamma'(t),\gamma'(t))}}{t}
    \]
and then defining the distance \(d(x,y)\) from \(x\in\oser(A)\) to \(y\in\oser(A)\) to be
    \[
        d(x,y)=\inf(\set[L(\gamma)]{\gamma(0)=x\text{ and }\gamma(1)=y}).
    \]

\section{Serenation}
\label{sec:serenation}
Now we have almost all the tools we require to complete the process analogous to desingularization of \(\simpcmplx(A)\). Since \(\oser(A)\) only includes the necessarily nonsingular points of \(\simpcmplx(A)\) (that is, those which do not belong to the \((n-2)\)-skeleton), we have already removed all singularities which may have been present. It remains to ``fill in holes'' in the appropriate fashion so that we obtain something which is topologically more interesting than a combinatorial graph.

The last ingredient is a slight modification of the usual notion of the completion of a metric space. Given a metric space \((S,d)\) let \(\cmplt(S,d)\) denote the set of all points in the metric completion of \((S,d)\). That is, \(\cmplt(S,d)\) is the set of all equivalences classes of Cauchy sequences of points in \(S\) under the equivalence relation induced by the metric \(d\). We endow \(\cmplt(S,d)\) with the metric topology induced by \(d\).

We say that a point \(x\) in a topological space \(T\) is \emph{\(n\)-Euclidean (in \(T\))} when there exists a neighborhood \(U\) of \(x\) which is homeomorphic to an open set in \(\R^n\).

\begin{defn}[Euclidean metric completion functor]
    We define a \emph{Euclidean metric completion functor}
        \[
            \eucmplt\colon\riem_n\to\mfld_n
        \]
    as follows. Given a Riemannian \(n\)-manifold \((M,g)\) consisting of a smooth \(n\)-manifold \(M\) and a Riemannian metric \(g\) whose corresponding metric on \(M\) is \(d\), define \(\eucmplt(M,g)\) to be the set of points
        \[
            \eucmplt(M,g)=\set[x\in\cmplt(M,d)]{x\text{ is }n\text{-Euclidean in }\cmplt(M,d)}
        \]
    equipped with the subspace topology inherited from \(\cmplt(M,d)\). Given a smooth map \(h\colon M_1\to M_2\) between Riemannian manifolds \((M_1,g_1)\) and \((M_2,g_2)\) which is a local isometry everywhere define
        \[
            \eucmplt(h)\colon\eucmplt(M_1,g_1)\to\eucmplt(M_2,g_2)
        \]
    by
        \[
            (\eucmplt(h))(\{x_1,x_2,x_3,\dots\}/\sim_1)=\{h(x_1),h(x_2),h(x_3),\dots\}/\sim_2
        \]
    where \(\{x_i\}_{i\in\PP}\) is a Cauchy sequence representing a member of \(\eucmplt(M_1,g)\) and \(\sim_1\) and \(\sim_2\) are the equivalence relations identifying Cauchy sequences with distance \(0\) from each other in \((M_1,g_1)\) and \((M_2,g_2)\), respectively.
\end{defn}

We give some motivating examples.

\begin{example}
Let \(S^n\) denote the \(n\)-sphere and let \(g\) be the usual metric on \(S^n\) inherited from \(\R^{n+1}\). Take \(M\) to be \(S^n\) minus a finite set of points. We have that \(\eucmplt(M,g)\cong S^n\).
\end{example}

The following example illustrates that in general
    \[
        \eucmplt(M,g)\neq\cmplt(M,d).
    \]

\begin{example}
Let \(T\) denote the \(2\)-dimensional torus and let
    \[
        C(T)=([0,1]\times T)/\sim
    \]
be a cone over \(T\). View \(C(T)\) as embedded in \(\R^5\) with
    \[
        T=\set[(\cos(\theta_1),\sin(\theta_1),\cos(\theta_2),\sin(\theta_2),0)]{\theta_1,\theta_2\in\R}
    \]
and cone point \((0,0,0,0,1)\). Take \(M\) to be \(C(T)\) minus the cone point \((0,0,0,0,1)\) and the points \(T\) of the original torus. We have that \(M\) is a smooth \(3\)-manifold which may be endowed with the usual metric inherited from \(\R^5\). We have that \(\eucmplt(M,g)=M\) while \(\cmplt(M,d)=C(T)\). The cone point \((0,0,0,0,1)\) is not included in \(\eucmplt(M,g)\) because this point of the cone over the torus is not \(3\)-Euclidean. To see this, note that if \((0,0,0,0,1)\) was \(3\)-Euclidean in \(C(T)\) then removing it would not change the fundamental group. Since the space obtained by removing the cone point from \(C(T)\) deformation retracts to \(T\) we would have that \(\pi_1(C(T))\cong\Z^2\). However, the cone over \(T\) is contractible so \(\pi_1(C(T))\) is trivial, a contradiction. The points of \(T\) are not included in \(\eucmplt(M,g)\) because they are also not Euclidean in the sense of having a neighborhood homeomorphic to an open set in \(\R^3\). We emphasize here that the points of \(T\) would be Euclidean in the sense of a manifold with boundary, as they do have neighborhoods homeomorphic to a half-space in \(\R^3\), but we do not allow that for our Euclidean metric completion.
\end{example}

\begin{defn}[Serenation functor]
We define a \emph{serenation functor}
    \[
        \ser_n\colon\ncaq_n\to\mfld_n
    \]
as follows. Given an alternating \(n\)-quasigroup \(A\) we define
    \[
        \ser(A)=\eucmplt(\oser(A),g)
    \]
where \(g\) is the standard metric on \(\oser(A)\). Given an NC homomorphism \(h\colon A\to B\) where \(A,B\in\aq_n\) we define
    \[
        \ser(h)=\eucmplt(h)
    \]
with respect to the the standard metrics on \(\oser(A)\) and \(\oser(B)\).
\end{defn}

Since every topological \(n\)-manifold for \(n=2\) or \(n=3\) carries a unique smooth structure we actually have functors which we might, by a slight abuse of notation, refer to as
    \[
        \ser_n\colon\ncaq_n\to\smfld_n
    \]
when \(n=2\) or \(n=3\).

It is perhaps too much to ask that any manifold \(M\) is of the form \(\ser(A)\) for some quasigroup \(A\), especially in light of the fact that any such manifold would have to be orientable and there exist nonorientable manifolds, but we would like to examine the slightly weaker condition that a connected orientable manifold \(M\) is a component of \(\ser(A)\) for some quasigroup \(A\).

\begin{defn}[Serene manifold]
We say that a connected orientable \(n\)-manifold \(M\) is \emph{serene} when there exists some alternating \(n\)-quasigroup \(A\) such that \(M\) is a component of \(\ser(A)\).
\end{defn}

This turns out to always be the case when \(M\) is triangulable. Our proof is constructive, provided a triangulation and orientation on the manifold in question.

\begin{thm}
\label{thm:orientable_triangulable}
Every connected orientable triangulable \(n\)-manifold is serene.
\end{thm}

\begin{proof}
We argue here the case where \(n>2\). The \(n=2\) case requires a slight modification, which we will indicate without giving the full details.

Let \(M\) be a connected orientable triangulable \(n\)-manifold and let \(\Gamma\) be a simplicial complex with \(\geo(\Gamma)=M\). We effectively subdivide \(\Gamma\) and use the orientation on \(M\) to define the requisite quasigroup \(A\).

For each facet \(\gamma=\set{s_1,\dots,s_{n+1}}\) of \(\Gamma\) fix an orientation
    \[
        \vec{\gamma}=(s_1,\dots,s_{n+1})/\alt_{n+1}
    \]
so that the complex \(\Gamma\) is oriented.

We define a sequence of sets \(A_i\) and a sequence of partial operations
    \[
        f_i\colon A_i^n\rightharpoonup A_i
    \]\
indexed over \(\N\). The desired quasigroup \(A\) will have its multiplication given by the union of these partial operations.

Define
    \[
        A_0=\vertices(\Gamma)\cup\fct(\Gamma)
    \]
and define a partial operation \(f_0\colon A_0^n\rightharpoonup A_0\) by
    \[
        f_0(s_1,\dots,s_n)=\gamma
    \]
when \(\gamma\) is the unique facet of \(\Gamma\) such that there exists some \(s_{n+1}\in\vertices(\Gamma)\) with \((s_1,\dots,s_{n+1})\in\vec{\gamma}\). We then set
    \[
        A_i^\times=\set[(a_1,\dots,a_n)/\alt_n\in A_i^n/\alt_n]{(a_1,\dots,a_n)\notin\dom(f_i)}
    \]
and let \(A_i^\div\) be the set of all pairs \(((x,a_2,\dots,a_n)/\alt_n,a_{n+1})\) where \(x\) is a variable, \(a_2,\dots,a_{n+1}\in A_i\), and there is no \(a_1\in A_i\) such that
    \[
        f_i(a_1,\dots,a_n)=a_{n+1}.
    \]
The set \(A_i^\times\) corresponds to all those \(n\)-ary products as yet undefined for \(f_i\) while the set \(A_i^\div\) corresponds to all those equations
    \[
        f_i(x,a_2,\dots,a_n)=a_{n+1}
    \]
which do not yet have a solution.

Observe that \(A_i^\div\) covers all equations where exactly one of the arugments of \(f\) is the unknown \(x\), for the action of the alternating group can move the \(x\) appearing in \((x,a_2,\dots,a_n)\) to any other slot. This is where a modification to the argument is needed for the \(n=2\) case. Since \(\alt_2\) is trivial we would need to introduce new elements corresponding to the left and right division operations. The remainder of the argument is otherwise similar.

We take
    \[
        A_{i+1}=A_i\cup A_i^\times\cup A_i^\div
    \]
and define
    \[
        f_{i+1}\colon A_{i+1}^n\rightharpoonup A_{i+1}
    \]
by setting
    \[
        f_{i+1}(a_1,\dots,a_n)=f_i(a_1,\dots,a_n)
    \]
when \((a_1,\dots,a_n)\in\dom(f_i)\),
    \[
        f_{i+1}(a_1,\dots,a_n)=(a_1,\dots,a_n)/\alt_n
    \]
when \((a_1,\dots,a_n)/\alt_n\in A_i^\times\), and
    \[
        f_{i+1}(u_1,\dots,u_n)=a_{n+1}
    \]
when \((u_1,\dots,u_n)/\alt_n\) is
    \[
        (((x,a_2,\dots,a_n)/\alt_n,a_{n+1}),a_2,\dots,a_n)/\alt_n
    \]
and \(a_2,\dots,a_{n+1}\in A_i\). That is, \(\dom(f_{i+1})\) consists of \(\dom(f_i)\) along with \(A_i^n\setminus\dom(f_i)\) and all tuples of the form \((u_1,\dots,u_n)\) where \((u_1,\dots,u_n)/\alt_n\) is
    \[
        (((x,a_2,\dots,a_n)/\alt_n,a_{n+1}),a_2,\dots,a_n)/\alt_n
    \]
and the values of \(f_{i+1}\) when applied to these arguments are as listed above. Note that when \(j\ge i\) the members of \(A_j^\times\) and \(A_j^\div\) can never belong to \(A_i\). 

Let \(A=\bigcup_{i\in\N}A_i\) and let \(f\colon A^n\to A\) be given by
    \[
        f(a_1,\dots, a_n)=f_i(a_1,\dots,a_n)
    \]
when \((a_1,\dots,a_n)\in\dom(f_i)\). We claim that \(f\) is an alternating \(n\)-quasigroup operation on \(A\).

We first show that \(f\) is indeed a function. Consider a tuple \((a_1,\dots,a_n)\in A^n\). By definition of \(A\) there exists a function \(\sigma\colon[n]\to\N\) such that \(a_i\in A_{\sigma(i)}\). Moreover, taking \(j=\max_i\sigma(i)\) we see that \((a_1,\dots,a_n)\in A_{j+1}^n\). Either \((a_1,\dots,a_n)\in\dom(f_j)\), in which case
    \[
        f(a_1,\dots,a_n)=f_j(a_1,\dots,a_n)
    \]
or \((a_1,\dots,a_n)\in A_j^n\) but \((a_1,\dots,a_n)\notin\dom(f_j)\) in which case
    \[
        f(a_1,\dots,a_n)=f_{j+1}(a_1,\dots,a_n)=(a_1,\dots,a_n)/\alt_n.
    \]
In any case we see that at least one value is assigned for \(f(a_1,\dots,a_n)\) for each \((a_1,\dots,a_n)\in A^n\).

It remains to show that \(f\) is well-defined. Define a family of sets \(D_i\) inductively by taking \(D_0=\dom(f_0)\) and
    \[
        D_{i+1}=\dom(f_{i+1})\setminus\dom(f_i)
    \]
We claim that \(A^n\) can be written as a disjoint union
    \[
        A^n=\coprod_{i\in\N}D_i.
    \]
To see this, consider a tuple \((a_1,\dots,a_n)\in A^n\). As before we choose a function \(\sigma\colon[n]\to\N\) such that \(a_i\in A_{\sigma(i)}\). This time we also require that \(\sigma(i)\) be the least natural number \(k\) such that \(a_i\in A_k\). Again we set \(j=\max_i\sigma(i)\) and we observe that \((a_1,\dots,a_n)\in A_j^n\) but \((a_1,\dots,a_n)\notin A_k^n\) for any \(k<j\). We therefore have that \((a_1,\dots,a_n)\notin\dom(f_k)\) for any \(k<j\) and hence \((a_1,\dots,a_n)\notin D_k\) for any \(k<j\). Either \((a_1,\dots,a_n)\in\dom(f_j)\), in which case \((a_1,\dots,a_n)\in D_j\) and \((a_1,\dots,a_n)\notin D_k\) for any \(k>j\), or \((a_1,\dots,a_n)\notin\dom(f_j)\), in which case \((a_1,\dots,a_n)\in\dom(f_{j+1})\) and hence \((a_1,\dots,a_n)\notin D_j\), \((a_1,\dots,a_n)\in D_{j+1}\), and \((a_1,\dots,a_n)\notin D_k\) for any \(k>j+1\). We find that each member of \(A^n\) belongs to exactly one of the \(D_i\) and it is evident by the definition of \(A\) and the \(f_i\) that each \(D_i\subset A^n\) so we will have that \(f\) is well-defined if we can show that its restriction to each \(D_i\) is well-defined.

When \((s_1,\dots,s_n)\in D_0\) then we have that \(f(s_1,\dots,s_n)\) is indeed assigned a unique value since we take \(f(s_1,\dots,s_n)\) to be the unique facet \(\gamma\) of \(\Gamma\) such that there exists some \(s_{n+1}\) such that \((s_1,\dots,s_{n+1})\in\vec{\gamma}\). When \((u_1,\dots,u_n)\in D_i\) for some \(i>0\) we have that either all of the \(u_i\) belong to \(A_{i-1}\) or at least one does not. In the case that all of the \(u_i\) belong to \(A_{i-1}\) we have that \((u_1,\dots,u_n)/\alt_n\in A_i^\times\) and
    \[
        f(u_1,\dots,u_n)=f_i(u_1,\dots,u_n)=(u_1,\dots,u_n)/\alt_n.
    \]
In the case that at least one of the \(u_i\) belongs to \(A_i\setminus A_{i-1}\) it must be that \((u_1,\dots,u_n)/\alt_n\) is
    \[
        (((x,a_2,\dots,a_n)/\alt_n,a_{n+1}),a_2,\dots,a_n)/\alt_n
    \]
where all the \(a_j\) belong to \(A_i\) and
    \[
        f(u_1,\dots,u_n)=f_i(u_1,\dots,u_n)=a_{n+1}.
    \]

It was permissible here to only consider the unique \(i\) for which \((a_1,\dots,a_n)\in D_i\) since whenever \(j>i\) we have by our inductive definition of the \(f_k\) that \(f_j(a_1,\dots,a_n)=f_i(a_1,\dots,a_n)\).

We now show that \(f\) satisfies the alternating axiom. Suppose that \((u_1,\dots,u_n)\) and \((v_1,\dots,v_n)\) both belong to \(A^n\) and
    \[
        (u_1,\dots,u_n)/\alt_n=(v_1,\dots,v_n)/\alt_n.
    \]
Observe that if \((u_1,\dots,u_n)\in D_i\) then \((v_1,\dots,v_n)\in D_i\) and
    \[
        f_i(u_1,\dots,u_n)=f_i(v_1,\dots,v_n)
    \]
since by definition the restriction of \(f_i\) to \(D_i\) is invariant under the action of the alternating group \(\alt_n\).

Consider the equation
    \[
        f(y,a_2,\dots,a_n)=a_{n+1}
    \]
where the \(a_i\in A\) and \(y\) is a variable. Again we let \(\sigma\colon\set{2,\dots,n}\to\N\) be such that \(\sigma(i)\) is the least natural number \(k\) such that \(a_i\in A_k\) and take \(j=\max_i\sigma(i)\). Either there is some \(a_1\in A_j\) such that \((a_1,\dots,a_n)\in\dom(f_j)\) and \(f_j(a_1,\dots,a_n)=a_{n+1}\), in which case \(y=a_1\) is a solution, or there is no such \(a_1\in A_j\), in which case \(y=((x,a_2,\dots,a_n)/\alt_n,a_{n+1})\) is a solution. It remains to establish the uniqueness of these solutions. Suppose that
    \[
        f_s(y,a_2,\dots,a_n)=a_{n+1}
    \]
has no solution for \(s<r\) and has at least one solution for \(s=r\). If \(r=0\) then there is a unique solution to this equation by the geometric definition of \(f_0\). If \(r>0\) then either \(a_{n+1}=(a_1,\dots,a_n)/\alt_n\) for some \(a_1\in A_{r-1}\) and the only possible solution is \(y=a_1\) or \(a_{n+1}\in A_{r-1}\) and the only possible solution is
    \[
        y=((x,a_2,\dots,a_n),a_{n+1}).
    \]
In any case, the first \(s\) for which the equation
    \[
        f_s(y,a_2,\dots,a_n)=a_{n+1}
    \]
has a solution only yields a single solution to that equation. In order to see that no new solutions to this equation may be introduced, again let \(r\) be the least value of \(s\) for which the equation in question has a solution. When \(s>r\) the only new products that are introduced yield elements of either \(A_s\) (products of tuples that belonged to \(A_{s-1}^n\) but were not in the domain of \(f_{s-1}\)) or \(A_{s-1}\) (when introducing solutions to equations which did not have a solution for \(f_{s-1}\)). The only way we could obtain the \(a_{n+1}\) under consideration (which must belong to either \(A_r\) or \(A_{r-1}\)) for \(s>r\) is when \(s=r+1\) and we are adding a solution to an equation. Only solutions to equations which did not previously have a solution can be introduced by our rule, however, so this cannot occur. There is thus exactly one solution to each equation of the form
    \[
        f(y,a_2,\dots,a_n)=a_{n+1}.
    \]

The argument concerning equations where the unknown \(y\) appears in a different coordinate are covered by the preceding one by the alternating property.

We claim that \(M\) is a component of \(\ser(A)\). To see this, first consider the simplicial complex
    \[
        \Gamma'=\bigcup_{\mathclap{\substack{\gamma\in\fct(\Gamma)\\s\in\gamma}}}\pow((\gamma\setminus\set{s})\cup\set{\gamma}).
    \]
Note that \(\geo(\Gamma')\cong\geo(\Gamma)\). The new vertices \(\gamma\) correspond to the products of vertices from \(\Gamma\) in \(A\) determined by the partial operation \(f_0\). Since \(M\) is connected so is \(\geo(\Gamma')\), which means that there is a component of \(\oser(A)\) which is homeomorphic to
    \[
        U=\geo(\Gamma')\setminus(\geo(\Gamma'))^{(n-2)},
    \]
which is the geometric realization of \(\Gamma'\) with its \((n-2)\)-skeleton excised. Since \(U\) is still connected we have that \(\eucmplt(U,g)\) (where \(g\) is the restriction of the standard metric on \(\oser(A)\) to \(U\)) is a component of \(\ser(A)\). Since \(U\) is all but the \((n-2)\)-skeleton of the manifold \(M\) we have that \(\eucmplt(U,g)\cong M\) and hence \(M\) is homeomorphic to a component of \(\ser(A)\).
\end{proof}

In the case of second countable smooth manifolds triangulation is always possible so we have the following corollary.

\begin{cor}
\label{cor:smooth_manifold}
Every connected orientable smooth manifold is serene.
\end{cor}

\begin{proof}
Every second countable smooth manifold can be triangulated \cite{whitehead}.
\end{proof}

The proof of the preceding theorem demonstrates that given a connected orientable triangulable manifold \(M\) a certain universal construction always yields an alternating quasigroup \(A\) for which \(M\) is homeomorphic to a component of \(\ser(A)\). Again let \(\Gamma\) be a simplicial complex with \(\geo(S)=M\) and fix an orientation of \(\Gamma\) as in the proof above. Let \(X=\vertices(\Gamma)\cup\fct(\Gamma)\) and take \(\free(X)\) to be the free alternating \(n\)-quasigroup generated by the set \(X\), as discussed briefly in \autoref{sec:alternating_quasigroups}. Define \(\mu_\Gamma\subseteq(\free(X))^2\) by
    \[
        \mu_\Gamma=\set[(\gamma,f(s_1,\dots,s_n))]{(\exists s_{n+1}\in\vertices(\Gamma))((s_1,\dots,s_{n+1})\in\vec{\gamma})}.
    \]
We set \(\theta_\Gamma=\cg^{\free(X)}(\mu_\Gamma)\) and define
    \[
        A=\free(X)/\theta_\Gamma.
    \]
This alternating quasigroup \(A\) is the most general alternating quasigroup presented by the relations \(\gamma=f(s_1,\dots,s_n)\) given in the binary relation \(\mu_\Gamma\).

It would be tempting to claim that by this general nonsense we have an example of an alternating quasigroup for which \(M\) is homeomorphic to a connected component of \(\ser(A)\), simplifying the proof of \autoref{thm:orientable_triangulable}. Unfortunately, the abstract existence of the quotient of \(\free(X)\) given by the presentation \(\mu_\Gamma\) does not suffice to prove the theorem, as it is possible that this quotient identifies together elements coming from \(X\) which we need to remain distinct. A priori we do not know whether, for example, the quotient \(\free(X)/\theta_\Gamma\) is commutative and hence yields an empty manifold for \(\ser(A)\).

In order to show that no such undesirable collapsing occurs we need to either prove abstractly using equational logic that no other nontrivial relations are induced by those in \(\mu_\Gamma\) or construct an explicit example of an alternating quasigroup which satisfies the desired relations but keeps all the elements of \(\vertices(\Gamma)\cup\fct(\Gamma)\) distinct, which is what we did in our proof of \autoref{thm:orientable_triangulable}. Since no relations can hold in \(\free(X)/\theta_\Gamma\) which do not hold in every model of the relations in \(\mu_\Gamma\) on the same set of generators \(\vertices(\Gamma)\cup\fct(\Gamma)\), our proof does show that we obtain the desired manifold \(M\) as a component of \(\free(X)/\theta_\Gamma\).

\subsection{Examples of serenation}
Our earlier work leaves little more to do in describing \(\ser(A)\) for those quasigroups which we have already visited in previous examples. We begin with the quaternion group considered in \autoref{ex:quaternion_open} and \autoref{ex:quaternion_graph}.

\begin{example}
Let \(G\) denote the quaternion group of order \(8\). By our analysis in \autoref{ex:quaternion_open} we have that \(\oser(G)\) consists of three \(2\)-spheres, each of which has had \(6\) points removed. It follows that \(\ser(G)\) is homeomorphic to the disjoint union of three \(2\)-spheres. This is the same space as the desingularized complex \(Y(Q_8)\) of Herman and Pakianathan \cite[p.18]{herman}.
\end{example}

The order \(5\) and \(6\) ternary quasigroups we've examined yield spheres as well.

\begin{example}
Let \(A\) be the order \(5\) alternating \(3\)-quasigroup from \autoref{ex:order_five}, \autoref{ex:order_five_open}, and \autoref{ex:order_five_graph}. Since \(\oser(A)\) consists of a \(3\)-sphere minus a \(1\)-dimensional subcomplex we find that \(\ser(A)\) is homeomorphic to the \(3\)-sphere.
\end{example}

\begin{example}
Let \(A\) be the order \(6\) alternating \(3\)-quasigroup from \autoref{ex:order_six}, \autoref{ex:order_six_open}, and \autoref{ex:order_six_graph}. Since \(\oser(A)\) consists of a \(3\)-sphere minus a \(1\)-dimensional subcomplex we find that \(\ser(A)\) is homeomorphic to the \(3\)-sphere.
\end{example}

\section{Compact manifolds and Latin cubes}
\label{sec:latin_cubes}
In general the construction in the proof of \autoref{thm:orientable_triangulable} may be expected to yield an infinite alternating quasigroup \(A\) such that \(M\) is a component of \(\ser(A)\). Necessarily this is the case when \(M\) is given with a triangulation \(\Gamma\) where \(\Gamma\) is infinite, for each member \(s\in\vertices(\Gamma)\) becomes a distinct generator of \(A\). In the case that \(\Gamma\) is a finite set (or, topologically speaking, in the case that \(M\) is compact) can we always take \(A\) to be finite?

\begin{defn}[Quasifinite manifold]
We say that a connected compact orientable smooth \(n\)-manifold \(M\) is \emph{quasifinite} when there exists a finite alternating \(n\)-quasigroup \(A\) such that \(M\) is homeomorphic to a component of \(\ser(A)\).
\end{defn}

\begin{prob}
\label{prob:finiteness_compactness}
Is every connected compact orientable triangulable manifold quasifinite?
\end{prob}

This problem's solution would be implied by a positive solution of another, more combinatorial, problem which has the flavor of a standard question in the theory of (binary) quasigroups at large. The prototype for that standard question is the Evans Conjecture, which states that every partial Latin square of order \(n\) with at most \(n-1\) entries may have its other entries filled in so as to obtain a complete Latin square. The Evans Conjecture was proven by Smetaniuk in 1981 \cite{smetaniuk}, and many similar results followed, including those on \(3\)-dimensional Latin cubes in \cite{kuhl}.

\begin{defn}[Partial Latin cube]
Given a set \(A\) and some \(n\in\PP\) we say that \(\theta\subseteq A^{n+1}\) is a \emph{partial Latin \(n\)-cube} when for each \(i\in[n+1]\) and each choice of
    \[
        a_1,\dots,a_{i-1},a_{i+1},\dots,a_{n+1}\in A
    \]
there exists at most one \(a_i\in A\) so that
    \[
        (a_1,\dots,a_{n+1})\in\theta.
    \]
\end{defn}

This is to say that a partial Latin cube is the graph of a partial function from \(A^n\) to \(A\) which satisfies the identities of an \(n\)-quasigroup wherever all the relevant operations are defined.

\begin{defn}[Partial alternating Latin cube]
Given a set \(A\) and some \(n\in\PP\) we say that \(\theta\subseteq A^{n+1}\) is a \emph{partial alternating Latin \(n\)-cube} when \(\theta\) is a partial Latin cube and for each \(\alpha\in\alt_n\) we have that if
    \[
        (a_1,\dots,a_n,b_1)\in\theta
    \]
and
    \[
        (a_{\alpha(1)},\dots,a_{\alpha(n)},b_2)\in\theta
    \]
then \(b_1=b_2\).
\end{defn}

\begin{defn}[Complete Latin cube]
We say that a partial Latin \(n\)-cube \(\theta\subseteq A^{n+1}\) is a \emph{complete Latin \(n\)-cube} when for each \(i\in[n+1]\) and each choice of
    \[
        a_1,\dots,a_{i-1},a_{i+1},\dots,a_{n+1}\in A
    \]
there exists at least one \(a_i\in A\) so that
    \[
        (a_1,\dots,a_{n+1})\in\theta.
    \]
\end{defn}

That is, a complete Latin \(n\)-cube is the graph of an \(n\)-quasigroup operation. We also refer to complete Latin \(n\)-cubes simply as \emph{Latin \(n\)-cubes}, the ``complete'' emphasizing the distinction from partial Latin cubes. Note that we might have defined a Latin cube without reference to the notion of a partial Latin cube as a relation \(\theta\subseteq A^{n+1}\) such that for each choice of
    \[
        a_1,\dots,a_{i-1},a_{i+1},\dots,a_{n+1}\in A
    \]
there exists a unique \(a_i\in A\) so that
    \[
        (a_1,\dots,a_{n+1})\in\theta.
    \]

\begin{defn}[Finite partial Latin cube]
We say that a partial Latin cube \(\theta\subseteq A^{n+1}\) is \emph{finite} when \(A\) is a finite set.
\end{defn}

Using this language, our Evans-like problem for alternating quasigroups may be stated as follows.

\begin{prob}
\label{prob:evans}
Given a finite partial alternating Latin cube \(\theta\subseteq A^{n+1}\) does there always exist a finite complete alternating Latin cube \(\psi\subseteq B^{n+1}\) such that \(\theta\subseteq\psi\)?
\end{prob}

This problem is a bit weaker than the Evans conjecture for the case \(n=2\), as we don't posit a relationship between \(\abs{\theta}\) and \(\abs{B}\). In the \(n=2\) situation the veracity of the Evans conjecture implies that the answer to \autoref{prob:evans} is ``yes'' for \(n=2\). This in turn implies that the answer to \autoref{prob:finiteness_compactness} is ``yes'' for \(n=2\). Thus, we have a corollary to the Evans conjecture.

\begin{cor}
Every connected compact orientable surface is a component of the serenation of some finite binary quasigroup.
\end{cor}

This result appears to be significantly easier to establish than if we required the binary quasigroup in question be a group, as Herman and Pakianathan did. They were only able to show that an infinite family of such surfaces occurred as components of the serenation of some finite group \cite[Corollary 3.5]{herman}. One may have asked the following question having only seen the construction of Herman and Pakianathan, but the preceding corollary adds weight to it.

\begin{prob}
Is every connected compact orientable surface a component of the serenation of some finite group?
\end{prob}

\printbibliography

\end{document}